\documentclass[11pt,a4paper]{article}
\usepackage[american]{babel}      

\usepackage{graphicx} 
\usepackage{amsmath}
\usepackage{mathrsfs}
\usepackage{amssymb}
\usepackage{xcolor}
\usepackage{mathtools}
\usepackage{amsthm}
\usepackage{comment}
\usepackage{float}

\theoremstyle{plain}
\newtheorem{theorem}{Theorem}[section]
\theoremstyle{plain}

\theoremstyle{plain}
\newtheorem{proposition}{Proposition}[section]
\theoremstyle{plain}
\newtheorem{lemma}{Lemma}[section]
\theoremstyle{plain}
\newtheorem{remark}{Remark}[section]
\theoremstyle{plain}

\theoremstyle{plain}

\newcommand{\mau}{\geq}
\newcommand{\miu}{\leq}

\newcommand{\N}{\mathbb{N}}
\newcommand{\R}{\mathbb{R}}

\newcommand{\be}{\begin{equation}}
\newcommand{\ee}{\end{equation}}

\title{A new class of positive linear operators preserving logarithmic functions}
\author{\textit{Laura Angeloni\thanks{Corresponding author}$\ ^1$, Danilo Costarelli$\ ^1$, Chiara Darielli$\ ^2$} \\
\\
$\ ^1$ Department of Mathematics and Computer Science \\
            University of Perugia\\
       1, Via Vanvitelli, 06123 Perugia, Italy    \\  \\
 $\ ^2$ Department of Mathematics and Computer Science\\ 
 University of Firenze\\ 67, Viale Morgagni, 50134 Firenze,  Italy
       \\  \\
  {\small {\tt laura.angeloni@unipg.it}} - {\small {\tt danilo.costarelli@unipg.it}}\\  {\small {\tt 
 chiara.darielli@unifi.it}} 
  }
\date{}

\begin{document}
\maketitle
\begin{abstract}
In this paper, we introduce a new class of positive linear operators that generalize the classical Bernstein operators. Specifically, we construct a sequence of operators that preserve the logarithmic function $\ln(1+\mu+x)$, with $\mu > 0$ and $x \in [0,1]$. We prove pointwise and uniform convergence and we derive a quantitative estimate of the approximation error in terms of the modulus of continuity. We also obtain a Voronovskaja-type asymptotic formula, that is used to establish saturation results and inverse theorems. In particular, the saturation class of the considered approximation process is characterized by solving a second order differential equation. Shape-preserving properties, such as monotonicity, concavity and variation diminishing, are also investigated. Finally, a simple application to signal denoising is addressed.

\end{abstract}
\vskip0.2cm


\vskip0.3cm

\noindent {\bf AMS Subjclass:} 41A60, 41A25, 41A30   

\section{Introduction}
The Bernstein polynomials, introduced by Sergej Bernstein in 1912, are defined for a function $f \in C([0,1])$ and for $n \in \mathbb{N}$ as
\begin{align*}
B_nf(x)=B_n(f,x) := \sum_{k=0}^{n} f\left(\frac{k}{n}\right) \binom{n}{k} x^k (1-x)^{n-k}, 
\end{align*} with $x \in [0,1]$.
Such operators provide an important tool for approximation and their properties play a crucial role in this context (see, e.g., \cite{ConApp}). 

Indeed, as it is well-known, they provide a constructive proof of the classical Weierstrass approximation theorem, which asserts that every continuous function on a closed and bounded interval can be uniformly approximated through polynomials. This result is a cornerstone of polynomial approximation, one of the main topics in Approximation Theory, and the Bernstein polynomials remain one of its most studied tools.

Over the years, many extensions of the Bernstein polynomials have been proposed in the literature, including, for example, the sampling-type operators (\cite{bardaro2017,bardaro2019}), the max-product operators (\cite{bede2016}), the neural network-type operators (\cite{kadak2022}), the Sz{\'a}sz-Mirakjan operators (\cite{acar2017szasz,acar2017gonska,gupta2018,Aral2019}), and many others (\cite{gupta2020,acar2020gamma,ozsarac2022,gupta2024new,gupta2024conv}).

Among these developments, an exponential generalization of the classical Bernstein polynomials was introduced by Aral, Cárdenas-Morales and Garrancho in~\cite{BerType}. This construction, known as Bernstein-type exponential polynomials, is a special case of a more general family of operators previously studied by Morigi and Neamtu in~\cite{Morigi}.

For a fixed real parameter $\mu > 0$, for $f\in C([0,1])$ and for $n \in \mathbb{N}$, these operators are defined as
\begin{align*} \displaystyle
    \mathscr{G}_nf(x) = \mathscr{G}_n(f,x):= \sum_{k=0}^n f\left( \frac{k}{n} \right) e^{-\mu k / n}e^{\mu x} p_{n,k}(\tilde a_n(x)),
\end{align*} where
\begin{align*}
    \tilde a_n(x)=\frac{e^{\mu x/n}-1}{e^{\mu/n}-1},\quad p_{n,k}(x):= \binom{n}{k} x^k (1-x)^{n-k},
\end{align*}
with $x \in [0,1]$. Such operators are naturally linked to the Bernstein polynomials $B_n$, since obviously 
\begin{align}\label{relexp}
    \mathscr{G}_n(f,x)=\exp_{\mu}(x)\,B_n\left(\frac{f}{\exp_\mu},\tilde a_n(x)\right),
\end{align}
where $\exp_\mu(x) := e^{\mu x}$. As a consequence, it can be shown that $\mathscr{G}_n$ preserves the exponential functions $\exp_\mu(x)$ and $\exp_\mu^2(x)$, i.e., for $x \in [0,1]$,
\begin{align*}
\mathscr{G}_n(\exp_\mu,x) = e^{\mu x}, \quad \mathscr{G}_n(\exp_\mu^2,x)= e^{2\mu x}.
\end{align*}
By this property, the operators $\mathscr{G}_n$ can be regarded as a particular case of the so-called King-type operators, introduced by King in \cite{king}.

For further results on approximation properties, shape-preserving characteristics, and simultaneous approximation, we refer to~\cite{Mond,Convexity,BerType,Acu2022a}. For extensions to the multidimensional case, see~\cite{Angeloni2025}.

Within this line of research, in the present paper we introduce a new class of positive linear operators, below denoted by \(\mathscr{L}_n\), that preserve the logarithmic function \(\ln_\mu(x) = \ln(1+\mu + x)\), with \(\mu > 0\) and \(x \in [0,1]\). The idea is to replace the exponential weights in $\mathscr{G}_n $ with suitable logarithmic ones. Moreover, the functions $\tilde a_n$ are replaced by suitable concave approximants of the identity function involving $\ln_{\mu}(x)$. Differently from the exponential type operators, in the literature there is a lack of results concerning operators that preserve logarithmic functions. We mention that some results in this direction can be found in \cite{ozsarac2023}, but by means of integral operators of Mellin type, therefore in a completely different setting. Besides this, a further motivation for the introduction of the operators \(\mathscr{L}_n\) is the fact that the properties of the logarithmic function could be useful to linearize nonlinear transformations of multiplicative type.

After the introduction of the proposed new operators, we investigate their main approximation properties. In particular, we establish pointwise and uniform convergence by means of a direct constructive technique. The uniform convergence can be faced alternatively using a classical Korovkin approach: in this setting, a Korovkin subset is provided by the powers 0,1,2 of the function $\ln_{\mu}(x)$. 

Furthermore, the relation with King-type operators is explored. This allows us to derive a quantitative estimate of the approximation error for the operators $\mathscr{L}_n$ and to formulate a Voronovskaja-type asymptotic result. Then, we use this formula to obtain saturation results and inverse theorems. In particular, by the Voronovskaja formula, a second order differential operator naturally arises: the set of the solutions of the corresponding homogeneous ODE  represents the saturation class of the considered operators. 
The proof of such result relies on the application of the generalized ''parabola technique'' provided by Garrancho and Cárdenas-Morales (\cite{GACA2010}).

To complete the study of the operators, we also examine some shape-preserving properties concerning monotonicity, concavity and variation diminishing. 

Finally, a "toy-model" related to signal denoising is presented: the idea is to linearize a Gaussian-type multiplicative noise employing a logarithmic transformation and then to apply the logarithmic operators to reconstruct the noise-free signal.

\section{Bernstein-type logarithmic operators}
Inspired by the exponential case, a new class of Bernstein-type operators relying on logarithmic transformations can be defined. The construction is based on replacing the exponential weights with logarithmic ones, leading to a different yet structurally similar approximation scheme. 

To this end, we consider the function 
\begin{align}\label{ln}
\ln_\mu(x):=\ln(1+\mu+x), \qquad x\in[0,1],
\end{align}
where $\mu >0$ is a fixed parameter. Now, for a bounded function $f:[0,1]\rightarrow \mathbb{R}$ and  $n \in \mathbb{N}$ we define the new operator as
\begin{align} \displaystyle \label{GnNew}
    \mathscr{L}_nf(x) = \mathscr{L}_n(f,x):=\ln_\mu(x)\sum_{k=0}^n f\left( \frac{k}{n} \right) \frac{1}{\ln_\mu\left(\frac{k}{n}\right)}\hspace{1mm}p_{n,k}(a_n(x)),
\end{align} where
\begin{align} \label{an}
    \displaystyle
a_n(x)\ :=\ \frac{\ln_\mu\left(\frac{x}{n}\right)-\ln\left(1+\mu\right)}{\ln_\mu\left(\frac{1}{n}\right)-\ln\left(1+\mu\right)}\ =\ \frac{\ln\left(1+\frac{x}{n(1+\mu)}\right)}{\ln\left(1+\frac{1}{n(1+\mu)}\right)},
\end{align} with $x\in[0,1]$. The function $a_n(x)$ is such that \( a_n(0) = 0 \) and \( a_n(1) = 1 \), is increasing and concave on \([0,1]\), and $a_n(x)\ge x$ for every $x\in [0,1]$.

\begin{lemma}\label{lemma_an}
Let \(a_n(x)\) be defined as in (\ref{an}). Then, the sequence \( (a_n(x))_n \) converges uniformly to the identity function \( x \) on \([0,1]\).
\end{lemma}
\begin{proof}
Since 
\begin{align*} \displaystyle
    (a_n(x)-x)'=\frac{1}{\ln\left(1+\frac{1}{n(1+\mu)}\right)}\frac{1}{1+\frac{x}{n(1+\mu)}}\frac{1}{n(1+\mu)}-1=0
\end{align*}
for
\begin{align*} \displaystyle
\bar{x}_n:=&n(1+\mu)\left[\frac{1}{\ln\left(1+\frac{1}{n(1+\mu)}\right)}\frac{1}{n(1+\mu)}-1\right]\\[0.5em]
=&\frac{1-n(1+\mu)\ln{\left(1+\frac{1}{n(1+\mu)}\right)}}{\ln\left(1+\frac{1}{n(1+\mu)}\right)},
\end{align*}
and $(a_n(x)-x)'>0$ if and only if $x<\bar{x}_n$, at the point $\bar{x}_n$ it is attained the maximum value of $(a_n(x)-x)$ in $[0,1]$, namely,
\begin{align*}\displaystyle
    a_n(\bar{x}_n)-\bar{x}_n &= 
    \frac{\ln\left(1+\frac{\bar{x}_n}{n(1+\mu)}\right)}{\ln\left(1+\frac{1}{n(1+\mu)}\right)}-\bar{x}_n\\[0.6em]
    &=\frac{1}{\ln{\left(1+\frac{1}{n(1+\mu)}\right)}}\left[\ln\left(\frac{1}{n(1+\mu)\ln{\left(1+\frac{1}{n(1+\mu)}\right)}}\right)-1\right]+\\[1mm]&\hspace{4mm}+n(1+\mu).
\end{align*}
Setting
\begin{align} \label{varn}
\varepsilon_n := \frac{1}{n(1+\mu)},
\end{align} 
there holds
\begin{align*} \displaystyle
\lim_{n\to+\infty}[a_n(\bar{x}_n)-\bar{x}_n] :=&\lim_{\varepsilon_n\rightarrow 0^+} F(\varepsilon_n)\\ =&\lim_{\varepsilon_n\rightarrow 0^+}\left\{ \frac{1}{\ln(1+\varepsilon_n)}\left[\ln\left(\frac{\varepsilon_n}{\ln(1+\varepsilon_n)}\right)-1\right] + \frac{1}{\varepsilon_n}\right\}.
\end{align*}
By the Taylor expansion of the logarithmic function and the geometric series, 
\[
\ln\left(\frac{\varepsilon_n}{\ln(1 + \varepsilon_n)} \right)=\ln\left(1 + \frac{\varepsilon_n}{2} + o(\varepsilon_n)\right)
= \frac{\varepsilon_n}{2} - \frac{\varepsilon_n^2}{8} + o(\varepsilon_n^2),
\]
therefore
\begin{align*}\displaystyle
F(\varepsilon_n)\, &=\, \frac{-1 + \frac{\varepsilon_n}{2} - \frac{\varepsilon_n^2}{8} + o(\varepsilon_n^2)}{\varepsilon_n - \frac{\varepsilon_n^2}{2} + o(\varepsilon_n^2)}+\frac{1}{\varepsilon_n}\,  =\,  \frac{o(\varepsilon_n^2)}{\varepsilon_n^2-\frac{\varepsilon_n^3}{2} + o(\varepsilon_n^3)}\to 0, \quad n\to +\infty,
\end{align*}
and so the sequence $(a_n(x))_n$ converges uniformly to $x$ in $[0,1]$.
\\
\end{proof}
We observe that \( \mathscr{L}_n \) are positive linear operators that preserve the logarithmic function \( \ln_\mu(x) \), defined in \eqref{ln}. Indeed,
\begin{align}\label{repr} \displaystyle
    \mathscr{L}_n(\ln_\mu,x)&=\ln_\mu(x)\sum_{k=0}^n \ln_\mu\left(\frac{k}{n}\right)\frac{1}{\ln_\mu\left(\frac{k}{n}\right)}\hspace{1mm}p_{n,k}(a_n(x))\notag
    \\&=\ln_\mu(x)\sum_{k=0}^n p_{n,k}(a_n(x))=\ln_\mu(x),
\end{align} since $\sum_{k=0}^n p_{n,k}(a_n(x))=1$, for every $x \in [0,1].$ Nevertheless we point out that, differently from the exponential operators $\mathscr{G}_n$, \(\mathscr{L}_n\) does not preserve the function $\ln_\mu^2(x)$. 
\\[2mm]
The sequence \((\mathscr{L}_n)_n\) is connected to the classical Bernstein polynomials. In particular, if we denote
\begin{align}\label{fmu}
f_\mu(x):=\frac{f(x)}{\ln_\mu(x)}, \qquad x\in[0,1],
\end{align}
there holds
\begin{align} \displaystyle \label{Rel}
    \mathscr{L}_nf(x) &= \ln_\mu(x)\sum_{k=0}^n f\left(\frac{k}{n} \right)\frac{1}{\ln_\mu\left(\frac{k}{n}\right)}p_{n,k}(a_n(x))\notag\\[0.5em]
    &= \ln_\mu(x) B_n\left( f_\mu,a_n(x)\right),
\end{align}
which is the $\ln_{\mu}-$analogue of \eqref{relexp}. 

\section{Pointwise and uniform convergence}
In this section, we first provide a constructive proof of the pointwise and uniform convergence of the operators \(\mathscr{L}_n\). 

\begin{theorem}\label{th-conv}
If $f:[0,1]\rightarrow \R$ is bounded, then 
\begin{align*} \displaystyle 
\lim_{n \rightarrow +\infty} \mathscr{L}_nf(x) =f(x),
\end{align*}
for every $x\in[0,1]$ where $f$ is continuous. Moreover, 
for every $f \in C([0,1])$ there holds
\begin{align*} \displaystyle 
\lim_{n \rightarrow +\infty} \mathscr{L}_nf(x) =f(x),\hspace{5mm}uniformly\hspace{3mm}for\hspace{3mm}x\in[0,1].
\end{align*}
\end{theorem}

\begin{proof}
We will prove the second part of the theorem, since the first one follows with analogous reasonings. \\
Let us fix $\varepsilon > 0$ and let us denote by $\delta=\delta(\varepsilon)$ the parameter of the uniform continuity of $f$ in $[0,1]$, correspondingly to $\varepsilon$, i.e.,
\begin{align*} \displaystyle
    \left|f(x)-f(y) \right|\leq \varepsilon,
\end{align*}
for every $x,y \in [0,1]$
with $|x-y|\leq\delta$. Then, we have
\begin{align*} \displaystyle
    \left|\mathscr{L}_n(f,x)-f(x)\right|&\le |\mathscr{L}_n(f,x)-f(x)\mathscr{L}_n(e_0,x)|+|f(x)\mathscr{L}_n(e_0,x)-f(x)|\\[2mm]&=:I_1+I_2,
\end{align*} where
\begin{align*} \displaystyle
I_1&=\left|\sum_{k=0}^n f\left(\frac{k}{n}\right) \frac{\ln_\mu(x)}{\ln_\mu\left(\frac{k}{n}\right)}\,p_{n,k}(a_n(x))-f(x)\sum_{k=0}^n \frac{\ln_\mu(x)}{\ln_\mu\left(\frac{k}{n}\right)}\,p_{n,k}(a_n(x))\right|\end{align*}\begin{align*}
&\le \sum_{k=0}^n \Big|f\left(\frac{k}{n}\right)-f(x)\Big|\frac{\ln_\mu(x)}{\ln_\mu\left(\frac{k}{n}\right)}\,p_{n,k}(a_n(x))\\[2mm]
&\le \Big\{\displaystyle\sum_{\left|x-\frac{k}{n}\right|\le \delta}+\displaystyle\sum_{\left|x-\frac{k}{n}\right|>\delta}\Big\}
\Big|f\left(\frac{k}{n}\right)-f(x)\Big|\frac{\ln_\mu(x)}{\ln_\mu\left(\frac{k}{n}\right)}\,p_{n,k}(a_n(x))
\\[2mm]&=:I_{1,1} +I_{1,2}.
\end{align*}
We note that 
\begin{align*} \displaystyle
    I_{1,1}\,\,\le\,\,\frac{\varepsilon}{\ln(1+\mu)}\ln_\mu(x)\sum_{\left|x-\frac{k}{n}\right|\le \delta}\, p_{n,k}(a_n(x)) \le \frac{\varepsilon}{\ln(1+\mu)}\,\ln(2+\mu).
\end{align*}
Moreover, we also have that
\begin{align*} 
    I_{1,2}\,\,&\le\,\,\frac{2\Vert f\Vert_{\infty}\ln_\mu(x)}{\ln(1+\mu)}\sum_{\left|x-\frac{k}{n}\right|>\delta} p_{n,k}(a_n(x))\\
    &\le\,\,\frac{2\Vert f\Vert_{\infty}\ln(2+\mu)}{\ln(1+\mu)}\sum_{\left|x-\frac{k}{n}\right|>\delta} p_{n,k}(a_n(x))\frac{\left|x-\frac{k}{n}\right|}{\left|x-\frac{k}{n}\right|}\\
    &\le\,\,\frac{2\Vert f\Vert_{\infty}\ln(2+\mu)}{\delta\ln(1+\mu)}\sum_{\left|x-\frac{k}{n}\right|>\delta}\Big|x-\frac{k}{n}\Big|\,\,p_{n,k}(a_n(x)),
\end{align*}
where
\begin{align*} 
    &\sum_{\left|x-\frac{k}{n}\right|>\delta} \Big|x-\frac{k}{n}\Big|\,\,p_{n,k}(a_n(x))\\
    &\hspace{15mm}\le\sum_{\left|x-\frac{k}{n}\right|>\delta} |x-a_n(x)|\,\,p_{n,k}(a_n(x))+\sum_{\left|x-\frac{k}{n}\right|>\delta} \Big|a_n(x)-\frac{k}{n}\Big|\,\,p_{n,k}(a_n(x)).
\end{align*}
Now, using the classical estimate for the first absolute moment of the Bernstein polynomials, i.e., 
\begin{align*}
\sum_{k=0}^n \left|y-\frac{k}{n}\right|p_{n,k}(y)<\frac{1}{2\sqrt{n}},
\end{align*}
$y\in [0,1]$ (see \cite{Angeloni2024}), it follows that
\begin{align*} \displaystyle
    \sum_{\left|x-\frac{k}{n}\right|>\delta} \left|a_n(x)-\frac{k}{n}\right|p_{n,k}(a_n(x))<\frac{1}{2\sqrt{n}}<\varepsilon\delta,
\end{align*} for $n$ sufficiently large. 
By Lemma \ref{lemma_an}, we have that, for $n$ large enough, $|a_n(x)-x|<\varepsilon\delta$, for every $x \in [0,1]$. Therefore,
\begin{align*}\displaystyle
    I_1&\le I_{1,1} +I_{1,2}\le \frac{\ln(2+\mu)}{\ln(1+\mu)}\varepsilon+\frac{4\Vert f\Vert_\infty\ln(2+\mu)}{\delta\ln(1+\mu)}\varepsilon\delta\\
    &=\varepsilon\ \frac{\ln(2+\mu)}{\ln(1+\mu)}\ (1+4\Vert f\Vert_\infty),
\end{align*} for $n$ sufficiently large. On the other side,
\begin{align*}\displaystyle
I_2&= |f(x)\mathscr{L}_n(e_0,x)-f(x)|\le \Vert f\Vert_\infty|\mathscr{L}_n(e_0,x)-e_0(x)|
\\[1.5mm]&=\Vert f\Vert_\infty\left|\sum_{k=0}^n\frac{\ln_\mu(x)}{\ln_\mu\left(\frac{k}{n}\right)}p_{n,k}(a_n(x))-1\right|
\\[1.5mm]&=\Vert f\Vert_\infty\left|\sum_{k=0}^n\left[\frac{\ln_\mu(x)}{\ln_\mu\left(\frac{k}{n}\right)}-\frac{\ln_\mu\left(\frac{k}{n}\right)}{\ln_\mu\left(\frac{k}{n}\right)}\right]p_{n,k}(a_n(x))\right|\\[1.5mm]
    &\le\Vert f\Vert_\infty\sum_{k=0}^n\frac{|\ln_\mu(x)-\ln_\mu\left(\frac{k}{n}\right)|}{\ln_\mu\left(\frac{k}{n}\right)}\,\,p_{n,k}(a_n(x))\\[1.5mm]
    &\le\Vert f\Vert_\infty\Big\{\sum_{\left|x-\frac{k}{n}\right|\le \delta}+\sum_{\left|x-\frac{k}{n}\right|>\delta}\Big\}\frac{|\ln_\mu(x)-\ln_\mu\left(\frac{k}{n}\right)|}{\ln(1+\mu)}\,\,p_{n,k}(a_n(x))\\[1.5mm]
    &=:I_{2,1}+I_{2,2},
\end{align*}
when $\delta$ is the parameter for the uniform continuity of $\ln_\mu(x)$ in $[0,1]$, correspondingly to $\varepsilon$.
By similar reasonings as before, we have that
\begin{align*}\displaystyle
    I_{2,1}\le\Vert f\Vert_\infty\frac{\varepsilon}{\ln(1+\mu)}\sum_{\left|x-\frac{k}{n}\right|\le \delta}\, p_{n,k}(a_n(x))\le\varepsilon\Vert f\Vert_\infty\frac{1}{\ln(1+\mu)}
\end{align*}
and
\begin{align*}
    I_{2,2}&\le \Vert f\Vert_\infty\frac{2\ln(2+\mu)}{\delta\ln(1+\mu)}\sum_{\left|x-\frac{k}{n}\right|>\delta}\left|x-\frac{k}{n}\right|p_{n,k}(a_n(x))\\[1mm]&\le\Vert f\Vert_\infty \frac{2\ln(2+\mu)}{\delta\ln(1+\mu)}\,2\delta\varepsilon
    =\varepsilon\Vert f\Vert_\infty\frac{4\ln(2+\mu)}{\ln(1+\mu)},
\end{align*} whence
\begin{align*}\displaystyle
    I_2\le I_{2,1}+I_{2,2}\le \varepsilon\Vert f\Vert_\infty\left(\frac{1}{\ln(1+\mu)}+\frac{4\ln(2+\mu)}{\ln(1+\mu)}\right).
\end{align*}
In conclusion, we obtain that, for $n$ sufficiently large, and for every $x\in[0,1]$,
\begin{align*}\displaystyle
     \left|\mathscr{L}_n(f,x)-f(x)\right|\le\varepsilon\ \frac{\ln(2+\mu)}{\ln(1+\mu)}\ \left[1+\Vert f\Vert_\infty\left(8+\frac{1}{\ln(2+\mu)}\right)\right], 
\end{align*}
from which the thesis immediately follows.\\
\end{proof}

We point out that the second part of the previous convergence result can be alternatively proved by means of the classical Korovkin theorem, which provides a sufficient condition to establish the uniform convergence of sequences of positive linear operators. Specifically, if $(T_n)_{n \in\N}$ is a sequence of positive linear operators on $C([0,1])$ such that $T_n(e_i) \to e_i$ uniformly on $[0,1]$ as $n \to +\infty$, with $e_i(x):=x^i$, $i=0,1,2$, then $T_n(f) \to f$ uniformly for every $f \in C([0,1])$.

This idea can be extended by recalling the well-known concept of Korovkin subset (\cite{AltoCam,Altomare2010}). For details and insights on Korovkin-type approximation theory, see, e.g. \cite{Bohman,Korovkin1953,AltoCam, Altomare2010} and the references therein. We can prove the following.

\begin{proposition}\label{Prop}
Given $\lambda_1,\lambda_2 \in \mathbb{R}$, $0<\lambda_1<\lambda_2,$ then
$\{1,\ln_{\mu}^{\lambda_1},\ln_{\mu}^{\lambda_2}\}$ is a Korovkin subset of $C([0,1])$.
\end{proposition}
\begin{proof} We consider $0<\lambda_1<\lambda_2$, $x_0 \in[0,1]$ and we define
\begin{align*}\displaystyle
    h(x):=1+\frac{\lambda_2}{\lambda_1-\lambda_2}\frac{\ln_\mu^{\lambda_1}(x)}{\ln_\mu^{\lambda_1}(x_0)}+\frac{\lambda_1}{\lambda_2-\lambda_1}\frac{\ln_\mu^{\lambda_2}(x)}{\ln_\mu^{\lambda_2}(x_0)},
\end{align*} with $h(x_0)=1+\frac{\lambda_2-\lambda_1}{\lambda_1-\lambda_2}=0$. Let now $x \neq x_0$ be fixed: then we have
\begin{align*}\displaystyle
    h'(x)&=\frac{\lambda_1\lambda_2}{\lambda_1-\lambda_2}\frac{1}{\ln_\mu^{\lambda_1}(x_0)}\frac{\ln_\mu^{\lambda_1-1}(x)}{(1+\mu+x)}+\frac{\lambda_1\lambda_2}{\lambda_2-\lambda_1}\frac{1}{\ln_\mu^{\lambda_2}(x_0)}\frac{\ln_\mu^{\lambda_2-1}(x)}{(1+\mu+x)}\\[0.5em]
    &\hspace{-5mm}=\frac{\lambda_1\lambda_2}{\lambda_2-\lambda_1}\frac{\ln_\mu^{-1}(x)}{(1+\mu+x)}\frac{\ln_\mu^{\lambda_1}(x)}{\ln_\mu^{\lambda_1}(x_0)}\left[-1+\frac{\ln_\mu^{\lambda_2-\lambda_1}(x)}{\ln_\mu^{\lambda_2-\lambda_1}(x_0)}\right].
\end{align*} The first derivative of $h(x)$ is equal to zero if and only if $x=x_0$ and is positive if and only if 
$$
\frac{\ln_\mu^{\lambda_2-\lambda_1}(x)}{\ln_\mu^{\lambda_2-\lambda_1}(x_0)}>1, \quad \text{i.e.},\quad x>x_0.
$$
 Then, $x_0$ is the point of absolute minimum of $h(x)$ and $h(x)>h(x_0)=0$, $\forall x\neq x_0$. 

Therefore, by Proposition 6.4 of \cite{Altomare2010}, we conclude that $\{1,\ln_{\mu}^{\lambda_1},\ln_{\mu}^{\lambda_2}\}$ is a Korovkin subset of $C([0,1])$.
\\\end{proof} 

In the proof of Theorem \ref{th-conv}, we saw that $(\mathscr{L}_n \ln^2_{\mu})_n$ converges uniformly to $\ln^2_{\mu}$ and, in a similar way it can be proved that $(\mathscr{L}_n 1)_n$ converges uniformly to the constant function $1$ on $[0,1]$. Therefore, taking into account of (\ref{repr}) and the fact that, by Proposition 6.4 of \cite{Altomare2010}, $\{1,\ln_{\mu},\ln^2_{\mu}\}$ is a Korovkin subset of $C([0,1])$, we can conclude that $(\mathscr{L}_n f)_n$ converges uniformly to $f$ on $[0,1]$, whenever $f\in C([0,1])$.

\section{Further approximation properties} 
In this section, we explore the approximation properties of the operators \(\mathscr{L}_n\), using their connection with the classical King-type operators (\cite{king,gonska2009king,finta2023}). In particular, we derive a quantitative estimate for the approximation error and we conclude with a Voronovskaja-type asymptotic result.
\\[0.5em]The King operators (\cite{king}) are defined as
\begin{align}\displaystyle\label{Vnf}
    \mathcal{V}_n f(x)= \mathcal{V}_n (f,x) := \sum_{k=0}^n f\left(\frac{k}{n}\right) \binom{n}{k} r_n^k(x) [1 - r_n(x)]^{n-k}, \quad x \in [0,1],
\end{align}
where $r_n(x)$ is continuous on $[0,1]$ such that $0 \le r_n(x) \le 1$. In~\cite{king}, J. King proved that
\begin{align}\displaystyle
    &\mathcal{V}_n(e_0, x) = 1, \label{Vne0}\\[0.5em]
    &\mathcal{V}_n(e_1, x) = r_n(x), \label{Vne1}\\[0.3em]
    &\mathcal{V}_n(e_2, x) = \frac{r_n(x)}{n} + \frac{n-1}{n} r_n^2(x). \label{Vne2}
\end{align}
Furthermore, he showed that $\mathcal{V}_n (f,x) \to f(x)$ as $n \to +\infty$, for every $f \in C([0,1])$ and $x \in [0,1]$, if and only if $\lim\limits_{n \to +\infty} r_n(x) = x$. 

For a function $f\in C([0,1])$, we recall that the modulus of continuity of $f$ is defined by
\begin{align*}
\omega(f,\delta) := \sup_{\substack{|x - y| \le \delta \\ x,y \in [0,1]}} |f(x) - f(y)|, \quad \delta > 0.
\end{align*} 
In \cite{Devore1972} it is provided the following estimate for every positive linear operator $L_n$ on $C([0,1])$ and for every $f \in C([0,1])$:
\begin{align} \label{K1}
    |L_n (f,x) - f(x)| \le \, &\omega(f, \delta) \left[L_n (e_0,x) + \frac{1}{\delta} \sqrt{L_n (e_0,x)} \alpha_n(x)\right] \notag\\&+|f(x)| \cdot |L_n(e_0,x) - e_0(x)|,
\end{align}
where $\alpha_n^2(x) = L_n (\varphi_t,x)$, with $\varphi_t(x) = (t - x)^2$.

Moreover, as remarked in \cite{king}, it is easy to see that
\begin{align}\displaystyle\label{alpha^2_n}
\alpha^2_n(x)=L_n(e_2, x)-2x L_n(e_1, x)+x^2 L_n(e_0, x).
\end{align}
The new operators, introduced in (\ref{GnNew}), can be rewritten as
\begin{align}\label{LK}
    \mathscr{L}_nf(x)=\ln_\mu(x)\mathcal{V}_n(f_\mu,x),
\end{align} with $r_n(x) \equiv a_n(x).$

By means of this relation, it is possible to provide a quantitative estimate of the error of approximation for the operators $\mathscr{L}_n$. In this regard, we state the following theorem.
\begin{theorem}
    Let $f \in C([0,1])$ be fixed. Then
\begin{align} \displaystyle \label{Lnf-f}
\Vert \mathscr{L}_n f - f \Vert_\infty \leq \ln(2+\mu) \,\, \omega\left( f_\mu, \frac{1}{\sqrt{n}} \right)\left(2+\sqrt{n}\gamma_n\right),
\end{align} for $n \in \mathbb{N}$, where $\gamma_n:= \max\limits_{x\in [0,1]} (a_n(x)-x).$
\end{theorem}
\begin{proof} 
By (\ref{LK}), we have that
\begin{align*}\displaystyle
    S:=&\vert \mathscr{L}_nf(x)-f(x)\vert=\left\vert \ln_\mu(x)\mathcal{V}_n(f_\mu,x)-\ln_\mu(x)f_\mu(x)\right\vert\\[0.5em]
    \le&\ln_\mu(x)\,\left\vert \mathcal{V}_n(f_\mu,x)-f_\mu(x)\right\vert.
\end{align*} 
Applying the estimate (\ref{K1}) to $V_n$ and $f_\mu$, we obtain
\begin{align*}
    S &\le \ln_\mu(x)\Big\{ \omega(f_\mu,\delta) \left[\mathcal{V}_n(e_0,x)+\delta^{-1}\sqrt{\mathcal{V}_n(e_0,x)}\,\alpha_n(x)\right] \Big.\\[0.5em]
    &\quad\Big. + \left| f_\mu(x) \right| \left| \mathcal{V}_n(e_0,x) - e_0(x) \right| \Big\}\\[0.5em]
    &= \ln_\mu(x)\,\omega(f_\mu,\delta)\left[1+\delta^{-1}\alpha_n(x)\right],
\end{align*}
for $\delta >0$. Since, by (\ref{alpha^2_n}), and (\ref{Vne0}), (\ref{Vne1}), (\ref{Vne2}),
\begin{align*}\displaystyle    \alpha_n^2(x)&=\mathcal{V}_n(e_2,x)-2x\mathcal{V}_n(e_1,x)+x^2\mathcal{V}_n(e_0,x)\\[0.5em]
    &=\frac{a_n(x)}{n}+\frac{n-1}{n}a_n^2(x)-2xa_n(x)+x^2,
\end{align*}
we obtain the estimate
\begin{align*}\displaystyle
    \delta^{-1}\alpha_n(x)&=\frac{1}{\delta}\left\{\frac{a_n(x)}{n}+\frac{n-1}{n}a_n^2(x)-2xa_n(x)+x^2\right\}^{\frac{1}{2}}\\[0.5em]
    &=\frac{1}{\delta}\left\{\frac{a_n(x)}{n}-\frac{a_n^2(x)}{n}+(a_n(x)-x)^2\right\}^{\frac{1}{2}}\\[0.5em]
    &=\left\{\frac{a_n(x)-a_n^2(x)}{n\delta^2}+\frac{1}{\delta^2}\left({a_n(x)-x}\right)^2\right\}^{\frac{1}{2}}\\[0.5em]
    &\le\left\{\frac{\Vert a_n\Vert_\infty \max\limits_{x\in[0,1]} \vert 1-a_n(x)\vert}{n\delta^2}+\frac{1}{\delta^2}\max\limits_{x \in[0,1]}( a_n(x)-x)^2\right\}^{\frac{1}{2}}.
\end{align*}
Taking $\delta=\frac{1}{\sqrt{n}}$, we get
\begin{align*}
    \sqrt{n}\,\alpha_n(x)\le\left\{\Vert a_n\Vert_\infty\max_{x \in[0,1]} \vert 1-a_n(x)\vert+n\max_{x \in[0,1]} (a_n(x)-x)^2 \right\}^{\frac{1}{2}}.
\end{align*}
Denoting \( \gamma_n := \max\limits_{x\in [0,1]} (a_n(x) - x) \), so that
\begin{align*}
\gamma_n^2 = \max_{x\in [0,1]} (a_n(x) - x)^2,
\end{align*} 
we conclude that
\begin{align*}\displaystyle
    S \leq \ln_\mu(x) \,\, \omega\left( f_\mu, \frac{1}{\sqrt{n}} \right)\left\{1+\Vert a_n \Vert_\infty \Vert 1-a_n\Vert_\infty+n\gamma_n^2\right\}^{\frac{1}{2}}.
\end{align*}
Since obviously $\sqrt{a+b+c}\le\sqrt{a}+\sqrt{b}+\sqrt{c}$, for $a,b,c\ge 0$,
\begin{align*}\displaystyle
    S &\leq \ln_\mu(x) \,\, \omega\left( f_\mu, \frac{1}{\sqrt{n}} \right)\left(1+\sqrt{\Vert a_n \Vert_\infty \Vert 1-a_n\Vert_\infty}+\sqrt{n}\gamma_n\right)\\
    &\leq \ln_\mu(x) \,\, \omega\left( f_\mu, \frac{1}{\sqrt{n}} \right)\left(2+\sqrt{n}\gamma_n\right).
\end{align*} 
Therefore 
\begin{align*}
    \Vert \mathscr{L}_n f-f\Vert_\infty \leq \ln(2+\mu)\,\,\omega\left(f_\mu,\frac{1}{\sqrt{n}} \right)(2+\sqrt{n}\gamma_n),
\end{align*} and the proof is complete.\\
\end{proof}

\begin{remark}
\normalfont
For the sake of completeness, we analyze the asymptotic behavior of the quantity \(\sqrt{n}\gamma_n\), where
$\gamma_n = \displaystyle\max_{x\in [0,1]} (a_n(x) - x)$.
To this end, we recall the notation \eqref{varn}, i.e.,
\[ \varepsilon_n := \frac{1}{n(1+\mu)}, \] 
so that 
\[
a_n(x) = \frac{\ln(1 + x \varepsilon_n)}{\ln(1 + \varepsilon_n)}.
\]
Applying the Taylor expansion of the logarithmic function and of the geometric series we get, for $n$ sufficiently large,
\begin{align*} 
a_n(x) &= \frac{x\varepsilon_n - \frac{x^2 \varepsilon_n^2}{2} + o(\varepsilon_n^2)}{\varepsilon_n - \frac{\varepsilon_n^2}{2} + o(\varepsilon_n^2)}
= \frac{x - \frac{x^2 \varepsilon_n}{2} + o(\varepsilon_n)}{1 - \frac{\varepsilon_n}{2} + o(\varepsilon_n)}
\\&= \left( x - \frac{x^2 \varepsilon_n}{2} + o(\varepsilon_n) \right) \left(1 + \frac{\varepsilon_n}{2} + o(\varepsilon_n) \right) 
\\&= x + \varepsilon_n \left( \frac{x-x^2}{2}\right)-\frac{x^2\varepsilon_n^2}{4} + o(\varepsilon_n^2).
\end{align*}
Therefore
\begin{align}\label{lim_n(an-x)}
\lim_{n \to +\infty} n(a_n(x) - x) = \frac{x - x^2}{2(1+\mu)}.
\end{align} 
Since \(x \in [0,1]\), the function \(x - x^2\) attains its maximum at \(x = \frac{1}{2}\), then
\[
\max_{x \in [0,1]} (x - x^2) = \frac{1}{4}, 
\]
and so
\begin{align*}
    \lim_{n\to +\infty}n\gamma_n=\frac{1}{8(1+\mu)}.
\end{align*}
Therefore, we conclude that
\begin{align}\label{Ogamma}
    \sqrt{n}\gamma_n=\mathcal{O}\left(\frac{1}{\sqrt{n}}\right),\quad\text{as\hspace{2mm}}n\to+\infty.
\end{align}
This shows that the term \(\sqrt{n} \gamma_n\) is arbitrarily small, as $n \to +\infty$, thereby validating the effectiveness of the convergence rate provided in \eqref{Lnf-f}. 
\end{remark}
Now, we establish a Voronovskaja-type asymptotic formula for the operators $\mathscr{L}_n$.
\begin{theorem}\label{Vor} 
  If $f \in C^2([0,1])$, then
\begin{align*} \displaystyle 
\lim_{n \to +\infty} n\left[ \mathscr{L}_n(f,x) - f(x) \right] 
= \ln_\mu(x)\,\frac{x-x^2}{2}\left[\frac{f_\mu'(x)}{(1+\mu)}+ f_\mu''(x)\right]
\end{align*}  for $x\in[0,1]$.
\end{theorem}
\begin{proof}
Using the notation (\ref{fmu}), we can rewrite the operator as
\[
\mathscr{L}_n(f,x) = \ln_{\mu}(x) \sum_{k=0}^n f_\mu\left( \frac{k}{n} \right) p_{n,k}(a_n(x)).
\]
If $f \in C^2[0,1]$, then $f_\mu \in C^2([0,1])$ as well. We now fix a point \( x \in [0,1] \) and apply the second-order Taylor formula to \( f_\mu \) with Peano remainder. In particular, we obtain
\[
f_\mu\left( \frac{k}{n} \right) = f_\mu(x) + f_\mu'(x) \left( \frac{k}{n} - x \right) + \left( \frac{k}{n} - x \right)^2 \left[\frac{1}{2} f_\mu''(x) + \psi\left( \frac{k}{n} - x \right)\right],
\]
where \( \psi(y) \) is a bounded function for every \( y\) that tends to zero as \( y \to 0^+ \). Replacing this expansion into the operator, we get
\begin{align*}
\mathscr{L}_n(f,x) &= \ln_\mu(x) \sum_{k=0}^n \left[ f_\mu(x) + f_\mu'(x) \left( \frac{k}{n} - x \right) + \frac{1}{2} f_\mu''(x) \left( \frac{k}{n} - x \right)^2 \right. \\
&\qquad \left. + \left( \frac{k}{n} - x \right)^2 \psi\left( \frac{k}{n} - x \right) \right] p_{n,k}(a_n(x))\end{align*}\begin{align*}
\hspace{15mm}&=\,\ln_\mu(x) \left[ f_\mu(x) \sum_{k=0}^n p_{n,k}(a_n(x)) + f_\mu'(x) \sum_{k=0}^n \left( \frac{k}{n} - x \right) p_{n,k}(a_n(x)) \right. \\
&\qquad \left. + \frac{1}{2} f_\mu''(x) \sum_{k=0}^n \left( \frac{k}{n} - x \right)^2 p_{n,k}(a_n(x)) \right.\\
&\qquad\left.+ \sum_{k=0}^n \left( \frac{k}{n} - x \right)^2\, \psi\left( \frac{k}{n} - x \right) p_{n,k}(a_n(x)) \right].
\end{align*} Denoting the last term, namely the remainder, by \( R_n(x)\) for brevity, the operator becomes 
\begin{align*}
\mathscr{L}_n(f,x)
&=f(x)+\ln_\mu(x)\left[ f_\mu'(x) \sum_{k=0}^n \left( \frac{k}{n} - x \right) p_{n,k}(a_n(x)) \right. \notag\\
& \left.\quad +\frac{1}{2} f_\mu''(x) \sum_{k=0}^n \left( \frac{k}{n} - x \right)^2 p_{n,k}(a_n(x)) + R_n(x)\right],
\end{align*}
so that 
\begin{align}\label{Ln-f}
\mathscr{L}_n(f,x)-f(x)&=\ln_\mu(x)\left[ f_\mu'(x) \sum_{k=0}^n \left( \frac{k}{n} - x \right) p_{n,k}(a_n(x)) \right. \notag\\
&\left.\quad +\frac{1}{2} f_\mu''(x) \sum_{k=0}^n \left( \frac{k}{n} - x \right)^2 p_{n,k}(a_n(x)) + R_n(x)\right].
\end{align}
We now analyze the individual terms appearing in the expression above. First, we consider 
\begin{align}\label{ris1}
\sum_{k=0}^n \left( \frac{k}{n} - x \right) p_{n,k}(a_n(x))&= \sum_{k=0}^n \frac{k}{n}\,\, p_{n,k}(a_n(x)) - \sum_{k=0}^n x \,\,p_{n,k}(a_n(x))\notag\\
&=\mathcal{V}_n(e_1,x)-x=a_n(x)-x,
\end{align}
by the identity \eqref{Vne1} for the King-type operators, previously defined in \eqref{Vnf}, with \( r_n(x) \equiv a_n(x) \). Next, we consider the second-order moment
\begin{align}\label{ris2}
\sum_{k=0}^n \left( \frac{k}{n} - x \right)^2 p_{n,k}(a_n(x)) 
&= \sum_{k=0}^n\frac{k^2}{n^2}\,\, p_{n,k}(a_n(x)) -2 x \sum_{k=0}^n \frac{k}{n}\,\,p_{n,k}(a_n(x)) +x^2\notag\\[1mm]
&=\mathcal{V}_n(e_2,x)-2x\mathcal{V}_n(e_1,x)+x^2\notag\\[1mm]
&=\frac{a_n(x)}{n} + \frac{n-1}{n} a_n^2(x) - 2x a_n(x) + x^2,
\end{align}
applying the identities \eqref{Vne0}, \eqref{Vne1} and \eqref{Vne2} for the King-type operators, with \( r_n(x) \equiv a_n(x) \).
Now, multiplying both sides of \eqref{Ln-f} by $n$ and using the expressions derived in \eqref{ris1} and \eqref{ris2}, we obtain
\begin{align}\label{n(Ln-f)}
n\left[\mathscr{L}n(f,x)-f(x)\right]=&\ln_\mu(x) \bigg\{ f_\mu'(x) \cdot n(a_n(x) - x) \notag\\
&+ \frac{1}{2} f_\mu''(x) \Big[ a_n(x) + (n-1)a_n^2(x) - 2n x a_n(x) + n x^2 \Big] \notag\\&+ n R_n(x) \bigg\}\notag\\
=&\ln_\mu(x) \bigg\{ f_\mu'(x) \cdot n(a_n(x) - x)\notag\\
&+ \frac{1}{2} f_\mu''(x) \Big[ a_n(x)-a_n^2(x) + n(a_n(x)-x)^2\Big] \notag\\&+ n R_n(x) \bigg\}.
\end{align}
We now analyze the limit of each term in \eqref{n(Ln-f)} separately as \( n \to +\infty \). For the first-order term, we use the result \eqref{lim_n(an-x)}, which gives
\begin{align*}
f_\mu'(x)\cdot\lim_{n \to +\infty} n\left( a_n(x) - x \right) = f_\mu'(x)\left(\frac{x - x^2}{2(1+\mu)}\right).
\end{align*}
For the second-order term, we use Lemma~\ref{lemma_an} and \eqref{lim_n(an-x)}, from which \( n(a_n(x) - x)^2 \to 0 \), as $n \to +\infty$ uniformly with respect to $x\in [0,1]$. Hence, we have
\begin{align*}
\frac{1}{2} f_\mu''(x) \cdot& \lim_{n \to +\infty} 
\left[ a_n(x) - a_n^2(x) + n \left( a_n(x) - x \right)^2 \right] \\[1mm]
&= \frac{1}{2} f_\mu''(x) \left\{ \lim_{n \to +\infty} \left[ a_n(x) - a_n^2(x) \right] 
+ \lim_{n \to +\infty} \big[n \left( a_n(x) - x \right)^2\big] \right\} \\[1mm]
&= \frac{1}{2} f_\mu''(x) \left( x - x^2 \right).
\end{align*}
Now, we consider the last term \(nR_n(x)\), given by
\begin{align*}
    nR_n(x) = n \sum_{k=0}^n \left(\frac{k}{n} - x\right)^2 \psi\left(\frac{k}{n} - x\right) \, p_{n,k}(a_n(x)).
\end{align*}
Since the function \(\psi(y)\) is bounded and tends to zero as $y\to 0^+$, fixed \(\varepsilon > 0\), there exists \(\tilde{\delta} > 0\) such that, for every \(y\) with \(0<y \le \tilde{\delta}\), then \(|\psi(y)| < \varepsilon\). As a consequence, we can write
\begin{align*}
    n|R_n(x)| &\le n \Bigg\{ \sum_{\left|\frac{k}{n} - a_n(x)\right| \le \frac{\tilde{\delta}}{2}} + \sum_{\left|\frac{k}{n} - a_n(x)\right| > \frac{\tilde{\delta}}{2}} \Bigg\} \left( \frac{k}{n} - x \right)^2 \left| \psi\left( \frac{k}{n} - x \right) \right| \, \\[2mm]&\qquad \cdot\, p_{n,k}(a_n(x)) =: R_1 + R_2.
\end{align*}
To estimate \(R_1\), we observe that, for \(n\) sufficiently large, 
\[
\left| \frac{k}{n} - x \right| \leq \left| \frac{k}{n} - a_n(x) \right| + \left| a_n(x) - x \right| \leq \tilde{\delta},
\]
and so \(\left| \psi\left( \frac{k}{n} - x \right) \right| < \varepsilon\). Then, using the identity \eqref{ris2}, we obtain
\begin{align*}
    R_1 &\leq \varepsilon \cdot n \sum_{k=0}^n \left( \frac{k}{n} - x \right)^2 p_{n,k}(a_n(x)) \\
        &= \varepsilon \cdot n \left[ \frac{a_n(x)}{n} + \frac{n-1}{n} a_n^2(x) - 2x a_n(x) + x^2 \right]\\
        &= \varepsilon \left[ {a_n(x)} - a_n^2(x) +n(a_n(x) - x)^2\right].
\end{align*}
Now, by Lemma \ref{lemma_an}, we know that \(a_n(x)\) converges uniformly to $x$ and, by (\ref{Ogamma}), \(n(a_n(x) - x)^2\to 0\), as $n\rightarrow +\infty$, so the entire expression is bounded as \(n \to +\infty\). Therefore, there exists a constant \(D > 0\) such that, for sufficiently large $n\in\N,$
\[
R_1 \le \varepsilon \, D.
\]
Moreover, 
\begin{align*}
    R_2 &\le \| \psi \|_\infty\cdot n\ \sum_{\left|\frac{k}{n} - a_n(x)\right| > \frac{\tilde{\delta}}{2}}\left( \frac{k}{n} - x \right)^2 \,p_{n,k}(a_n(x)) \\
    &\le \| \psi \|_\infty\cdot n\sum_{\left|\frac{k}{n} - a_n(x)\right| > \frac{\tilde{\delta}}{2}} \left[2\left( \frac{k}{n} - a_n(x) \right)^2+2\left( a_n(x)-x\right)^2\right] \,p_{n,k}(a_n(x))\\
    &\le {2}\| \psi \|_\infty \left\{n\sum_{\left|\frac{k}{n} - a_n(x)\right| > \frac{\tilde{\delta}}{2}} \left( \frac{k}{n} - a_n(x) \right)^2 \,p_{n,k}(a_n(x))+n(a_n(x)-x)^2\right\},
\end{align*}
since obviously $(a+b)^2\le2a^2+2b^2$, for every $a,b \in \R$. To estimate the first term, notice that, by (1.6) in Chapter 10 of \cite{ConApp} (a corollary of Theorem 1.1 of \cite{ConApp}), for each \( \tilde{\delta} > 0 \)  there exists a constant $C=C({\tilde{\delta}})$, such that (for $s=r+1$ and $r=1$),
\[
\sum_{ \left| \frac{k}{n} - a_n(x) \right| > \frac{\tilde{\delta}}{2} } p_{n,k}(a_n(x)) \leq C\, n^{-2}.
\] 
Since \( \left( \frac{k}{n} - a_n(x) \right)^2 \le 1 \) for all \( k \le n \), it follows that
\begin{align*}
&{2}\| \psi \|_\infty \, n \sum_{ \left| \frac{k}{n} - a_n(x) \right| > \frac{\tilde{\delta}}{2}} \left( \frac{k}{n} - a_n(x) \right)^2 p_{n,k}(a_n(x)) 
\\& \le {2}\| \psi \|_\infty\, n \sum_{ \left| \frac{k}{n} - a_n(x) \right| > \frac{\tilde{\delta}}{2}} p_{n,k}(a_n(x)) \le {2}\| \psi \|_\infty C\, n^{-1} .
\end{align*}
The estimate of the second part is a consequence of Lemma \ref{lemma_an} and \eqref{lim_n(an-x)}, from which $n(a_n(x)-x)^2\to 0$, as $n \to +\infty.$ Therefore, for \(n\) sufficiently large, we have that 
$R_2\le A\| \psi \|_\infty\, \varepsilon$, for some $A>0$, and, consequently,
\begin{align*}
    n|R_n(x)|\le \varepsilon(D+A).
\end{align*}
By the results obtained for each term, we derive
\begin{align*}
\lim_{n \to +\infty} n\left[ \mathscr{L}_n(f,x) - f(x) \right] 
= \ln_\mu(x)\,\frac{x-x^2}{2}\left(\frac{f_\mu'(x)}{(1+\mu)}+ f_\mu''(x)\right).
\end{align*} This completes the proof.
\\\end{proof}

\section{Saturation results and inverse theorems}
From the previous Voronovskaja-type formula (Theorem \ref{Vor}), the following second order differential operator arises:
\be \label{diff-operator}
{\cal D}(f)(x)\, :=\ \frac12\, \ln_\mu(x)\, (x-x^2)\, \left[ \frac{1}{ \mu+1}f'_\mu(x)\, +\, f''_\mu(x) \right], \quad x \in [0,1],
\ee
where $f\in C^2([0,1])$. By simple calculations, it is easy to see that the above differential operator can be rewritten in the following standard form:
\be \label{abstract-form-D}
{\cal D}(f)\, =\, {1 \over w_2}\cdot \left(    {1 \over w_1} \left( {f \over w_0}  \right)^{'}  \, \right)^{'},
\ee
with
$$
w_0(x)\, :=\, \ln_\mu(x), \quad \quad \quad w_1(x)\, :=\, e^{-x/(\mu+1)}, 
$$
$$
w_2(x)\, :=\, {2 \over (x-x^2)\, \ln_\mu(x)}\, e^{x/(\mu+1)}, \quad \quad 0<x<1.
$$
The above functions $w_i$, $i=0,1,2$, are strictly positive and sufficiently regular in $(0,1)$, and they form an extended complete Tchebychev system and a fundamental set of solutions of the homogeneous equation $D(f)=0$ (see \cite{KAST1966}).
\\[0.5em]
In \cite{GACA2010} Garrancho and C\'ardenas-Morales, inspired by the 
so-called {\em parabola} {\em technique} introduced by Bajsanski and Bojanic in \cite{BABO1964} (see also Chapter 5 of \cite{Devore1972} for a more detailed explanation), established the following useful result involving differential operators of the form (\ref{abstract-form-D}).
\begin{lemma} \label{lemma-GCM}
Let $J$ be any fixed open interval on $[0,1]$. Let $g$, $h \in C(J)$ and $t_0 < t_1 < t_2$ in $J$, be such that $g(t_1)=g(t_2)=0$, and $g(t_0)>0$. Then, there exist $\alpha<0$, a function $z$ which is a classical solution of the second order differential equation $D(f)=0$ on $J$, and a point $x \in (t_1,t_2)$ such that:
$$
\alpha\, h(t)+z(t)\, \mau\, g(t), \quad t \in [t_1,t_2], 
$$
and
$$
\alpha\, h(x)\, +\, z(x)\, =\, g(x).
$$
\end{lemma}
Now, we can prove the following
\begin{theorem}
Let $f \in C([0,1])$ be fixed. Then
\be \label{ipotesi-o-piccolo}
\left| {\mathscr L}_n (f, x)-f(x) \right|\ =\ o(n^{-1}), \quad \quad 0<x<1,
\ee
if and only if $f$ is a classical solution of the following second order differential equation:
\be \label{eq-diff-saturation}
f_\mu'(x)+ (1+\mu)\, f_\mu''(x)\ =\ 0.
\ee
\end{theorem}
\begin{proof}
We prove the necessary part of the thesis, since the sufficient one immediately comes from the previous Voronovskaja-type theorem.
\\[1mm]
Let now $y:[0,1] \to \R$ be the classical solution of the following differential problem with boundary conditions:
$$
\begin{cases}
\displaystyle {\cal D}(y)(x)= 0, \quad 0<x<1, \\
y(0) = f(0), \\
y(1) = f(1),
\end{cases}
$$
where ${\cal D}$ is defined in (\ref{diff-operator}). 
\\[2mm]
Now, we define $g(x):=f(x)-y(x)$, $x \in [0,1]$. By construction, it results $g \in C([0,1])$, and $g(0)=g(1)=0$.
\\[1mm]
From this consideration we can deduce that at least one between the maximum or the minimum of $g$ is assumed in $t_0 \in (0,1)$, otherwise we must have $f\equiv y$ on $[0,1]$ and the thesis would be trivially satisfied.
\\[1mm]
Suppose for instance that $t_0$ is a maximum point, and assume by contradiction that $g(t_0)>0$ (the proof below is completely analogous if $t_0$ is a minimum; it is sufficient to repeat the following computations with $-g$ in place of $g$).
\\[2mm]
Let us now consider  any fixed solution $h:[0,1] \to \R$ of the following non-homogeneous differential problem:
$$
{\cal D}(h)(x)\ =\ 1,\ \quad 0\miu x \miu 1.
$$
By Lemma \ref{lemma-GCM} applied to $g$ and $h$ on $(0,1)$, we know that there exist $\alpha<0$, a function $z$ on $(0,1)$ such that ${\cal D}(z)=0$ and a point $x \in (0,1)$ such that
$$
G(t)\, :=\, \alpha\, h(t)\, +\, z(t)\, -\, g(t)\ \mau\ 0, \quad 0\miu t \miu 1,
$$
(where the above function $z$ can be considered extended as a $C^2$-function on $[0,1]$), and $G(x)=0$.
\\[1mm]
Recalling that ${\mathscr L}_n$ is a positive linear operator and $G$ is non-negative, we have ${\mathscr L}_n(G,x) \mau 0$, and thus using that $G(x)=0$ we can write what follows:
$$
n \left[{\mathscr L}_n(G,x)\, -\, G(x)\right]\, \mau 0,
$$
namely,
$$
n \left[{\mathscr L}_n(\alpha\, h\, +\, z\, -\, f\, +\, y,\, x)\, -\, \alpha\, h(x)\, -\, z(x)\, +\, f(x) - y(x)\right]\, \mau 0,
$$
from which we get
$$
n \left[{\mathscr L}_n(\alpha\, h\, +\, z\, +\, y,\, x)\, -\, \alpha\, h(x)\, -\, z(x)\, -\, y(x)\right]\, \mau\ n \left[{\mathscr L}_n(f,\, x)\, -\, f(x)\right].
$$
Now, passing to the limit for $n \to +\infty$ and using (\ref{ipotesi-o-piccolo}) and Theorem \ref{Vor} we obtain that
$$
{\cal D}(\alpha\, h\, +\, z\, +\, y)(x)\ = \alpha\, {\cal D}(h)(x)\, +\, {\cal D}(z)(x)\, +\, {\cal D}(y)(x)\, =\ \alpha \mau 0,
$$
from which we get a contradiction.
This completes the proof. \\
\end{proof}

For the sake of completeness, we can easily observe that solutions of the differential equations (\ref{eq-diff-saturation}) are of the form:
$$
f(x)\ =\ A\, \ln_\mu(x)\, +\, B\, \ln_\mu(x)\, e^{-x/(\mu+1)}, \quad 0<x<1, \quad A,\, B \in \R.
$$
Finally, also the following inverse theorem of approximation can be established; it immediately follows by a direct application of Proposition 2 of \cite{GACA2010}.
\begin{theorem}
Let $f \in C([0,1])$ be fixed. Then
\begin{align*}
n\left| {\mathscr L}_n (f, x)-f(x) \right|\ \miu\ M + o(1), \quad \quad 0<x<1, \quad as \quad n \to +\infty,
\end{align*}
if and only if, for almost every $t \in (0,1)$:
$$
|f_\mu'(t)+ (1+\mu)\, f_\mu''(t)|\ \miu\ M.
$$
\end{theorem}

\section{Shape-preserving properties}
In this section we study some shape preserving properties of the linear and positive operators $\mathscr{L}_n$.
\vspace{1mm} 

Since 
\begin{align*} \displaystyle
    p_{n,k}'(a_n(x)) = &\, \binom{n}{k}k(a_n(x))^{k-1}a_n'(x)(1-a_n(x))^{n-k} 
    \\ & -\binom{n}{k}(a_n(x))^k(n-k)(1-a_n(x))^{n-k-1}a_n'(x),
\end{align*}
by (\ref{Rel}) there holds
\begin{align}
\left( \frac{\mathscr{L}_n f}{\ln_\mu} \right)'(x) 
&= \left( B_n\left(f_\mu,a_n\right) \right)'(x) 
= \sum_{k=0}^n f_\mu\left( \frac{k}{n} \right) \, p_{n,k}'(a_n)(x) \notag \\
&= a_n'(x) \sum_{k=1}^n f_\mu\left( \frac{k}{n} \right) \binom{n}{k} \, k \, (a_n(x))^{k-1} (1 - a_n(x))^{n - k} \notag \\
&\quad - a_n'(x) \sum_{k=0}^{n-1} f_\mu\left( \frac{k}{n} \right) \binom{n}{k} (n - k) \, (a_n(x))^k (1 - a_n(x))^{n - k - 1} \notag \\
&= a_n'(x) \sum_{k=1}^n f_\mu\left( \frac{k}{n} \right) \frac{n!}{(k - 1)! (n - k)!} \, (a_n(x))^{k - 1} (1 - a_n(x))^{n - k} \notag \\
&\quad - a_n'(x) \sum_{k=0}^{n-1} f_\mu\left( \frac{k}{n} \right) \frac{n!}{k! (n - k - 1)!} \, (a_n(x))^k (1 - a_n(x))^{n - k - 1} \notag \\
&= n \, a_n'(x)\Bigg[ \sum_{k=0}^{n-1} f_\mu\left(\frac{k+1}{n}\right)\binom{n-1}{k}(a_n(x))^k(1-a_n(x))^{n-1-k} \notag \\
&\quad -\sum_{k=0}^{n-1} f_\mu\left(\frac{k}{n}\right)\binom{n-1}{k}(a_n(x))^k(1-a_n(x))^{n-1-k} \Bigg] \notag \\
&= n\, a_n'(x) \sum_{k=0}^{n-1} \left[f_\mu\left(\frac{k+1}{n}\right)-f_\mu\left(\frac{k}{n}\right)\right]p_{n-1,k}(a_n(x))\notag\\
&=  a_n'(x) \sum_{k=0}^{n-1} \Delta_1 f_\mu\left( \frac{k}{n} \right) p_{n-1,k}(a_n(x)), \label{der1}
\end{align}
where
\begin{align*}
\Delta_1 f_\mu\left( \frac{k}{n} \right) := \left[f_\mu\left( \frac{k+1}{n} \right) - f_\mu\left( \frac{k}{n} \right)\right]n
\end{align*}
denotes the divided difference of order 1 of $f_{\mu}$ (see \cite{ConApp}). Therefore,
\begin{align}\label{der1Lnf}
\mathscr{L}_n'(f,x) 
&= \frac{B_n\left(f_\mu, a_n(x)\right)}{1 + \mu + x} 
+ \,\ln_\mu(x)\,a_n'(x)\sum_{k=0}^{n-1} \Delta_1 f_\mu\left( \frac{k}{n} \right) p_{n-1,k}(a_n(x)).
\end{align}
For the second derivative, we get
\begin{align*}
     \left(\frac{\mathscr{L}_n f}{\ln_\mu}\right)''(x)=& {d\over dx}\left(n\, a_n'(x) \sum_{k=0}^{n-1} \left[f_\mu\left(\frac{k+1}{n}\right)-f_\mu\left(\frac{k}{n}\right)\right]p_{n-1,k}(a_n(x))\right)\\[1em]
     =&  a_n''(x) S(x)+ a_n'(x)S'(x),
\end{align*} where 
\begin{align*}
    S(x):=\sum_{k=0}^{n-1} \Delta_1 f_\mu\left( \frac{k}{n} \right)\,\,p_{n-1,k}(a_n(x)),
\end{align*} and 
\begin{align*}
    S'(x) &= \sum_{k=0}^{n-1} \Delta_1 f_\mu\left( \frac{k}{n} \right) p'_{n-1,k}(a_n(x)) \\
    &= a_n'(x) \left[ \sum_{k=1}^{n-1} \Delta_1 f_\mu\left( \frac{k}{n} \right) \binom{n-1}{k} k \, a_n(x)^{k-1} (1 - a_n(x))^{n-1-k} \right. \\
    &\quad \left. - \sum_{k=0}^{n-2} \Delta_1 f_\mu\left( \frac{k}{n} \right) \binom{n-1}{k} (n-1-k) a_n(x)^k (1 - a_n(x))^{n-2-k} \right] \\
    &= n(n-1) a_n'(x) \sum_{k=0}^{n-2} \left[ f_\mu\left( \frac{k+2}{n} \right) - 2 f_\mu\left( \frac{k+1}{n} \right) + f_\mu\left( \frac{k}{n} \right) \right]\\&\quad\cdot p_{n-2,k}(a_n(x))\\
    &=\frac{n-1}{n} a_n'(x) \sum_{k=0}^{n-2} \Delta_2 f_\mu\left(\frac{k}{n}\right)\, p_{n-2,k}(a_n(x)), 
\end{align*}
where
\begin{align*}
\Delta_2 f_\mu\left(\frac{k}{n}\right) := \left[f_\mu\left(\frac{k+2}{n}\right) - 2 f_\mu\left(\frac{k+1}{n}\right) + f_\mu\left(\frac{k}{n}\right)\right] n^2
\end{align*}
denotes the divided difference of order 2 of $f_{\mu}$. 
Therefore, the second derivative of ${\mathscr{L}_n f}/{\ln_\mu}$ is
\begin{align*} 
    \left(\frac{\mathscr{L}_n f}{\ln_\mu}\right)''(x)= & \hspace{1mm}  a_n''(x)
    \sum_{k=0}^{n-1}\Delta_1 f_\mu\left( \frac{k}{n} \right)\,\, p_{n-1,k}(a_n(x)) \notag \\
    & +\frac{n-1}{n} a_n'(x)^2 \hspace{1mm} \sum_{k=0}^{n-2} \Delta_2 f_\mu\left( \frac{k}{n} \right)\,\,p_{n-2,k}(a_n(x)).
\end{align*}
Since \(a_n(x)\) is increasing  on the interval \([0,1]\), it follows from expressions~\eqref{der1} and~\eqref{der1Lnf} that, if $f_{\mu}(x)$ is increasing on $[0,1]$, both \({\mathscr{L}_n f}/{\ln_\mu}\) and \(\mathscr{L}_nf\) are increasing on \([0,1]\). And if in addition $f_\mu(x)$ is concave then, taking into account that $a_n(x)$ is concave,  \({\mathscr{L}_n f}/{\ln_\mu}\) is concave as well.\\

Now, in addition to the shape-preserving properties discussed above, we show that the operator $\mathscr{L}_n$ is bounded with respect to the $BV_\mu$-norm. Let us recall that $AC([0,1])$, the space of absolutely continuous functions on the interval $[0,1]$, is a closed subspace of the space $BV_{\mu}([0,1])$, the space of all the functions of bounded variation on $[0,1]$ with respect to the seminorm
\begin{align*} \displaystyle
    Var_{[0,1]}\left(\frac{f}{\ln_\mu}\right)=Var_{[0,1]} (f_\mu),
\end{align*}
where
\begin{align*}
    Var_{[0,1]} (f_\mu) := \sup_{P\in \mathcal{P}} \sum_{i=0}^{m-1} \left| f_\mu(x_{i+1}) - f_\mu(x_i) \right|,
\end{align*}
and \(\mathcal{P}\) denotes the family of all the partitions of the interval \([0,1]\), that is, sets of the form \(P = \{0 = x_0 < x_1 < \ldots < x_m = 1\}\).
Furthermore, we recall that $BV_{\mu}([0,1])$ can be endowed with the norm
\begin{align*} \displaystyle
    \Vert f \Vert_{BV_\mu} := Var_{[0,1]}\left(\frac{f}{\ln_\mu} \right)+|f(0)|.
\end{align*}
Obviously, $BV_{\mu}([0,1])$ coincides with $BV([0,1])$, the classical space of the functions of bounded variation in the sense of Jordan on $[0,1].$

We now proceed to establish the boundedness of the operator $\mathscr{L}_n$ with respect to the $BV_{\mu}$-norm. 
\begin{theorem}
    If $f \in BV([0,1])$, then for all $n \in \mathbb{N}$ we have that
    \begin{align*} \displaystyle
        \Vert \mathscr{L}_n f \Vert_{BV_\mu} \leq \Vert f \Vert_{BV_\mu}.
    \end{align*}
\end{theorem}
\begin{proof}

First, notice that the total variation of $\mathscr{L}_nf/\ln_\mu$ can be expressed as an integral over $[0,1]$. Indeed, by (\ref{der1}), 
\begin{align*}
\left|\left(\frac{\mathscr{L}_nf}{\ln_\mu}\right)'(x)\right|   \le {2\over \mu+1} \left[\ln\left(1+\frac{1}{n(1+\mu)}\right)\right]^{-1} \Vert f_{\mu}\Vert_{\infty}
\end{align*}
Now, since the expression \eqref{der1} holds, we obtain that 
\begin{align*} \displaystyle
Var_{[0,1]}\left( \frac{\mathscr{L}_n f}{\ln_\mu}\right) & = \int_0^1 \left| \left( \frac{\mathscr{L}_n f}{\ln_\mu} \right)'(x)\right| dx \\
& = \int_0^1 n \hspace{1mm} a_n'(x) \sum_{k=0}^{n-1} \left[f_\mu\left(\frac{k+1}{n}\right)-f_\mu\left(\frac{k}{n}\right)\right]p_{n-1,k}(a_n(x))dx
\\& \leq n\sum_{k=0}^{n-1} \left|f_\mu\left(\frac{k+1}{n}\right)-f_\mu\left(\frac{k}{n}\right)\right| \int_0^1 p_{n-1,k}(a_n(x)) a_n'(x)dx \\
& \leq \sum_{k=0}^{n-1} \left|f_\mu\left(\frac{k+1}{n}\right)-f_\mu\left(\frac{k}{n}\right)\right|\leq     Var_{[0,1]}\left(f_\mu\right).
\end{align*} On the other hand \begin{align*} \displaystyle
\left|\left| {\mathscr{L}_n f}\right|\right|_{BV_\mu} & = Var_{[0,1]}\left(\frac{\mathscr{L}_n f}{\ln_\mu} \right)+|\mathscr{L}_n (f,0)| \\ &
\leq Var_{[0,1]}\left(\frac{f}{\ln_\mu} \right)+|f(0)| = \left|\left| f \right|\right|_{BV_\mu},
\end{align*} whence the thesis.\\
\end{proof}
We now show that the operators $\mathscr{L}_n$ satisfy a monotonicity property, similar to a relation previously obtained in~\cite{gonska2009king} for King-type operators.
\begin{proposition}
(a) If $f_{\mu}(x)$ is increasing and convex, then, for every $n \in \mathbb{N}$ and $x \in ]0,1[$,
\[
\mathscr{L}_{n}(f,x) \ge \mathscr{L}_{n+1}(f,x) \ge f(x).
\]
(b) If $f_{\mu}(x)$ is decreasing and concave, then, for every $n \in \mathbb{N}$ and $x \in ]0,1[$,
\[
\mathscr{L}_{n}(f,x) \le \mathscr{L}_{n+1}(f,x) \le f(x).
\]
\end{proposition}

\begin{proof}
We will prove part $(a)$ since part $(b)$ follows immediately applying the same considerations of $(a)$ to the function $-f(x)$, so that $-f_{\mu}(x)$ will be increasing and convex.

We recall that the new operator $\mathscr{L}_n$ is related to the classical Bernstein operators through the relation~\eqref{Rel}, namely,
\[
\mathscr{L}_n(f,x) = \ln_\mu(x) \cdot B_n\left(f_\mu, a_n(x)\right).
\]
Since $f_\mu(x)$ is convex on $[0,1]$, as a consequence of the classical shape preserving properties of Bernstein operators (see~\cite{ConApp}), we have that
\begin{align*}
    B_n(f_\mu,y) \ge B_{n+1}(f_\mu,y), 
\end{align*}
for every $y \in [0,1]$. Moreover, since $a_{n}(x) \ge a_{n+1}(x)$ for every $x \in [0,1]$, $n\in\N$, and $f_\mu(x)$ is increasing, it follows that
\begin{align*}
    B_{n+1}(f_\mu,a_n(x)) \ge B_{n+1}(f_\mu,a_{n+1}(x)).
\end{align*}
Combining these, we obtain
\begin{align*}
\mathscr{L}_n(f,x) - \mathscr{L}_{n+1}(f,x)
&= \ln_\mu(x)\left[B_n(f_\mu,a_n(x)) - B_{n+1}(f_\mu,a_{n+1}(x))\right] \\[0.5em]
&= \ln_\mu(x)\left[B_n(f_\mu,a_n(x)) - B_{n+1}(f_\mu,a_n(x)) \right. \\[0.5em]
& \quad\left. + B_{n+1}(f_\mu,a_n(x)) - B_{n+1}(f_\mu,a_{n+1}(x)) \right] \ge 0,
\end{align*}
namely $\mathscr{L}_n(f,x) \ge \mathscr{L}_{n+1}(f,x)$ for all $n\in\N$.  

Finally, $\mathscr{L}_n(f,x) \to f(x)$, as $n\to+\infty$, on every continuity point $x$ of $f(x)$ (Theorem \ref{th-conv}), therefore on every  $x\in ]0,1[$:  indeed, $f_{\mu}(x)$ is continuous in $]0,1[$, since it is convex, and so is $f(x)=f_{\mu}(x)\cdot \ln_{\mu}(x)$. Therefore, since 
 $\mathscr{L}_n(f,x)$ is decreasing with respect to $n$, we deduce that 
\begin{align*}
\mathscr{L}_{n}(f,x) \ge \mathscr{L}_{n+1}(f,x) \ge f(x), \quad \text{for all } n \in \mathbb{N}, \quad x\in ]0,1[.
\end{align*}\end{proof}

\section{Final remarks and future developments: a simple application to signal denoising}

In this paper, we introduced a new definition of positive linear operators, which is characterized by the preservation of the function $\ln_\mu(x)$. For such operators, we established a number of approximation properties, such as uniform convergence, Voronovskaja-type formula, saturation and inverse theorems based on the solution of a suitable differential problem, and shape preserving properties.

This new constructive approximation tool can be viewed as a reconstruction process based on suitable powers involving logarithmic functions. 

The main properties of ${\mathscr L}_n$ discussed in the previous sections suggest the possibility of deducing a denoising algorithm for signals affected by multiplicative noise.

In real world situations, it can occur that a given signal $f$ is characterized by the presence of multiplicative noise, such as Gaussian multiplicative noise, uniform multiplicative noise, and several others. For instance, an example of multiplicative noise is the so-called speckle (e.g., \cite{Aubert, Seelamantula2015}), that typically affects remote sensing data, or tomography and ultrasound echography images. Other cases in which multiplicative noise arises is in stochastic models for finance or in wireless signal applications.

Below, we present a very simple possible application of the logarithmic operators ${\mathscr L}_n$ for denoising; the main purpose of the following computations is to illustrate a possible future development of the above approximation tools.
\vskip0.2cm

Let $f:[0,1]\to \R^+$ be a given signal modelling a physical quantity. Suppose that $f$ is sampled (by a certain measurement tool) at a certain time $t>0$ in the uniformly spaced nodes $k/n$, $k=0,1,...,n$, $n \in \N$, obtaining a sequence of measured sample values $(y_{k,n})_k$.

We now suppose that $f$ is affected by a multiplicative noise function of the form: $n(x,t)=1+\mu(t)+x$, where the term $\mu(t)$ will be now viewed as a Gaussian random variable of mean $0$ and variance $\sigma^2$, namely 
$$
\mu(t) \sim {\cal N}(0,\sigma^2).
$$
Hence, we expect that the sampled values $(y_{k,n})_k$, measured at time $t>0$, are of the form:
$$
y_{k,n} \approx \left( 1+\mu(t)+ \frac{k}{n}\right)\, f\left(\frac{k}{n}\right).
$$
Let now consider $t>0$ for which $\mu(t)>0$ and define $g(x):=\left( 1+\mu(t)+ x\right) f(x)$; thus
$$
\ln (g(x))\ =\ \ln_{\mu(t)}(x)\, +\, \ln (f(x)).
$$
Applying now the logarithmic operators ${\mathscr L}_n$ with $\mu(t)$ in place of the previously used parameter $\mu$, we get, by (\ref{repr}):
$$
{\mathscr L}_n(\ln g,x)\, =\ \ln_{\mu(t)}(x)\, +\, {\mathscr L}_n(\ln f,x), \quad x \in [0,1].
$$
From the above relation, it is very simple to deduce the following denoising formula:
\be \label{denoising-method-2}
f(x)\ \approx\ e^{{\mathscr L}_n(\ln f,x)} \ =\ {\displaystyle e^{{\mathscr L}_n(\ln g,x)} \over \displaystyle 1 + \mu(t) + x },
\ee
where in (\ref{denoising-method-2}) we replaced ${\mathscr L}_n(\ln g,x)$ with its approximated version computed by means of the sampled values $y_{k,n}$, namely
\begin{align*}
{\mathscr L}_n(\ln g,x)\ \approx\ \ln_{\mu(t)}(x)\sum_{k=0}^n \frac{\ln (y_{k,n})}{\ln_{\mu(t)}\left(\frac{k}{n}\right)}\hspace{1mm}p_{n,k}(a_n(x)).
\end{align*}
In conclusion, the final denoising formula is:
\be \label{denoising-method}
f(x)\ \approx\ {1  \over \displaystyle 1 + \mu(t) + x }\, \exp\left[ \ln_{\mu(t)}(x)\sum_{k=0}^n \frac{\ln (y_{k,n})}{\ln_{\mu(t)}\left(\frac{k}{n}\right)}\hspace{1mm}p_{n,k}(a_n(x))    \right].
\ee
We now present a numerical example. Let us consider, for instance, three values of $\mu(t)$ randomly generated as a Gaussian random variable ${\cal N}(0,\sigma^2)$, with $\sigma=0.5$
$$
\mu_1 := \mu(t_1) = 0.2688, \quad   \mu_2 := \mu(t_2) = 0.9169, \quad  \mu_3 := \mu(t_3) = 1.1294, 
$$
and let $f:[0,1]\to\R^+$ defined as
$$
f(x)\ :=\ \frac15\, x^2 + \sin(x)+ \frac12 x+ \frac{1}{10}.
$$
We now consider the following noise-free reconstructions of $f$, for $n=10,30$ in Fig.s \ref{fig1}, \ref{fig2} and \ref{fig3}, implemented by MATLAB software. 
In the captions it is also reported the maximum reconstruction error computed between the original signal $f$ and the reconstructed one by means of (\ref{denoising-method}) on the considered nodes. 

\begin{figure}[H]
\centering
\includegraphics[scale=0.27]{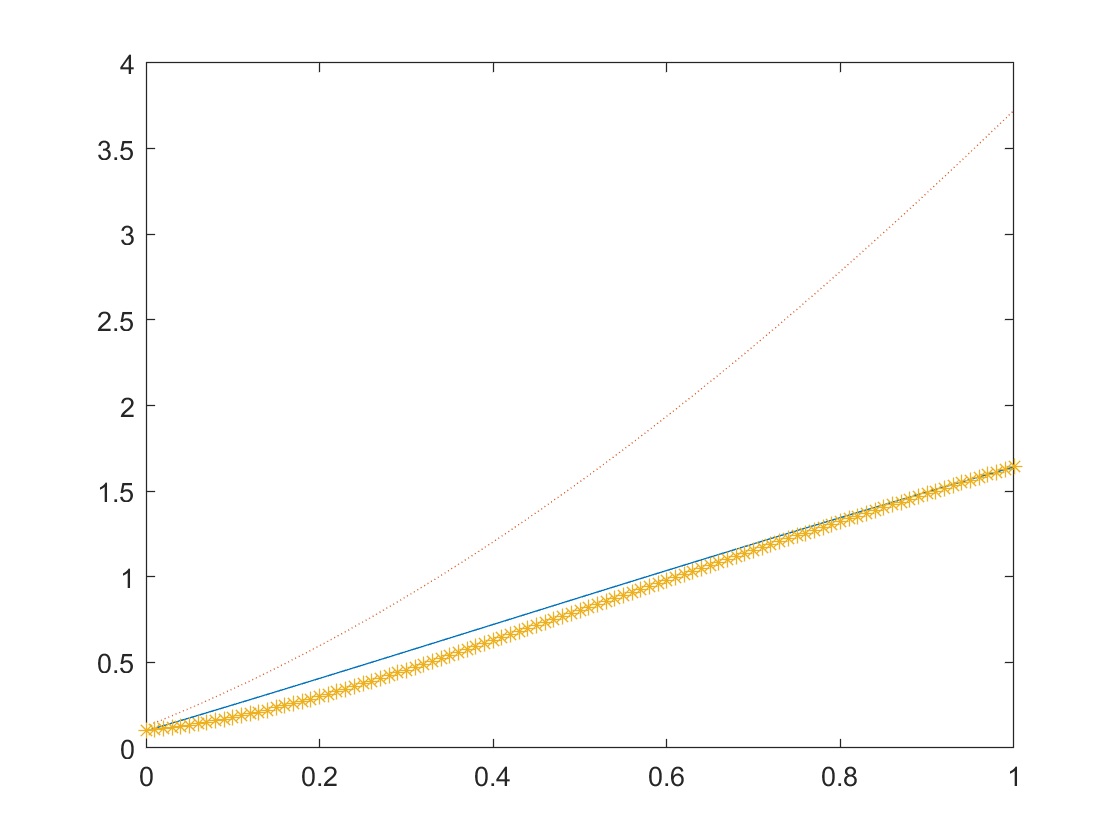}
\\[-1.5em]
\includegraphics[scale=0.27]{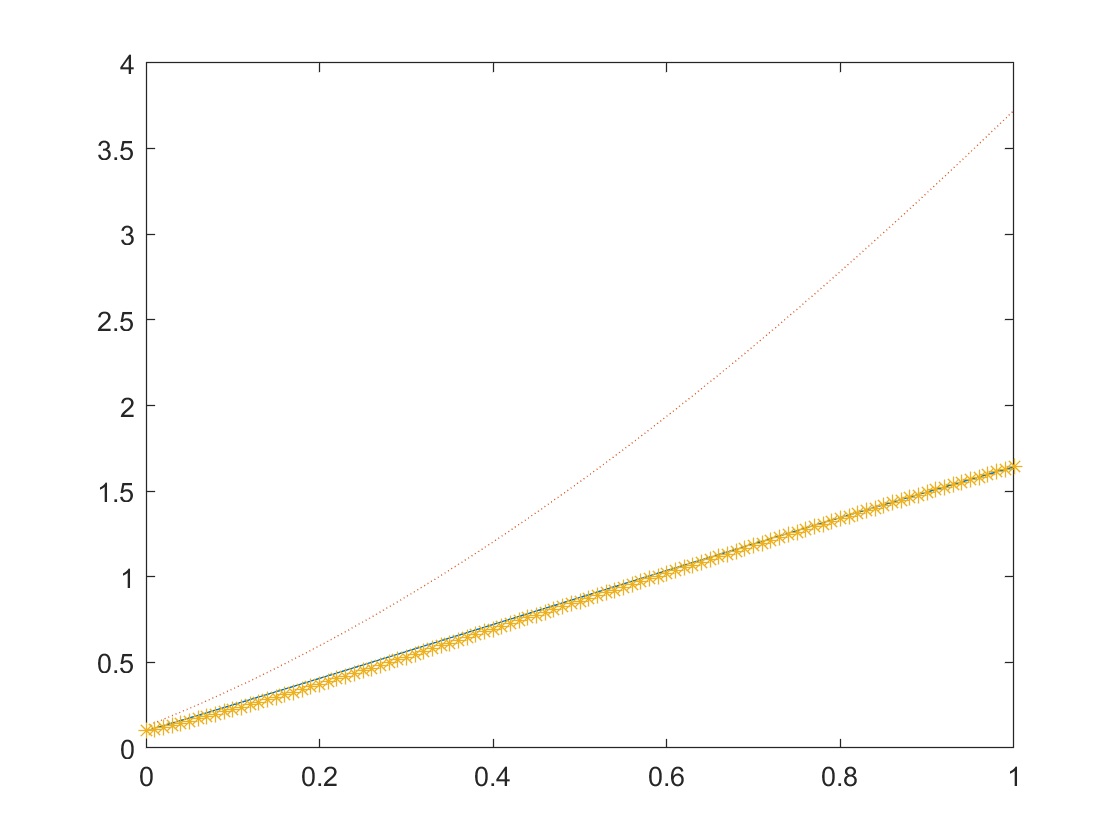}
\\[-1em]
\caption{{\small The (blue) solid line represents the original signal $f$, the dotted (red) plot represents the signal $g$ affected by the multiplicative noise, the asterisk (yellow) plot represent the denoised signal by (\ref{denoising-method}) for $\mu_1$, with $n=10$ (on the top) and $n=30$ (on the bottom). The maximum reconstruction errors are $0.1109$ and $0.0343$, respectively.}} \label{fig1}
\end{figure}
\begin{figure}[H]
\centering
\includegraphics[scale=0.27]{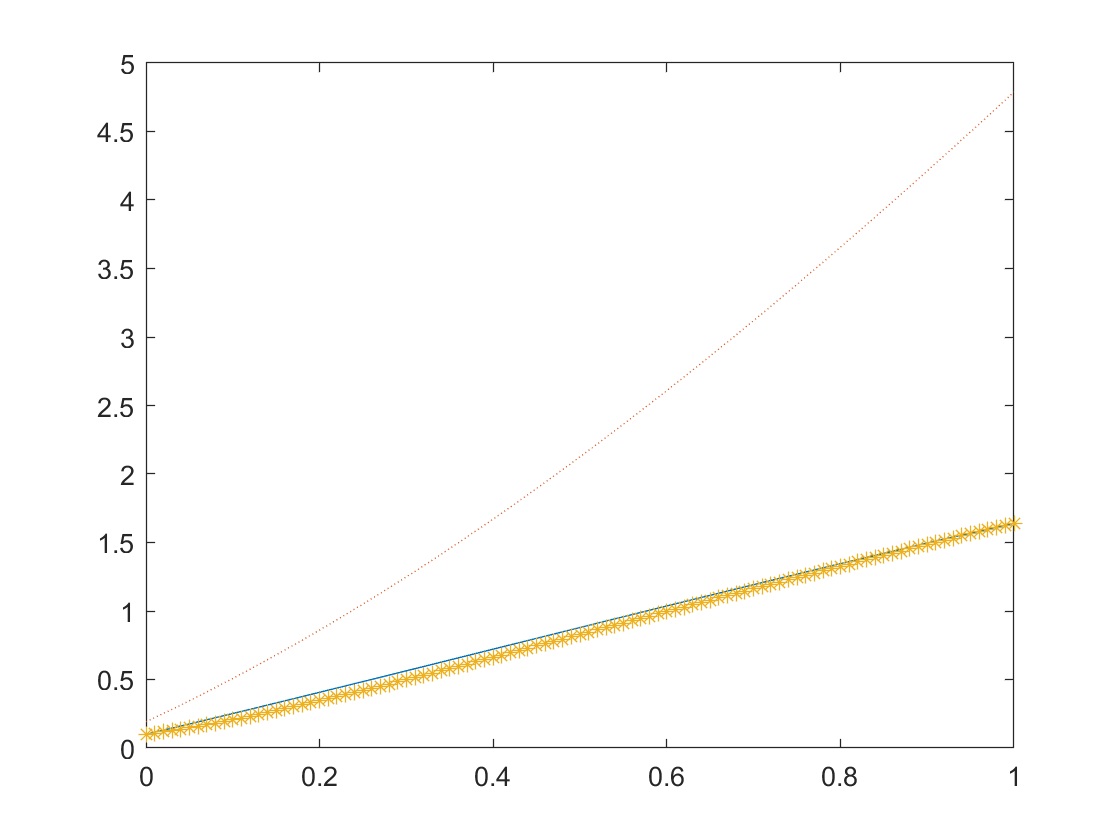}
\\[-1.5em]
\includegraphics[scale=0.27]{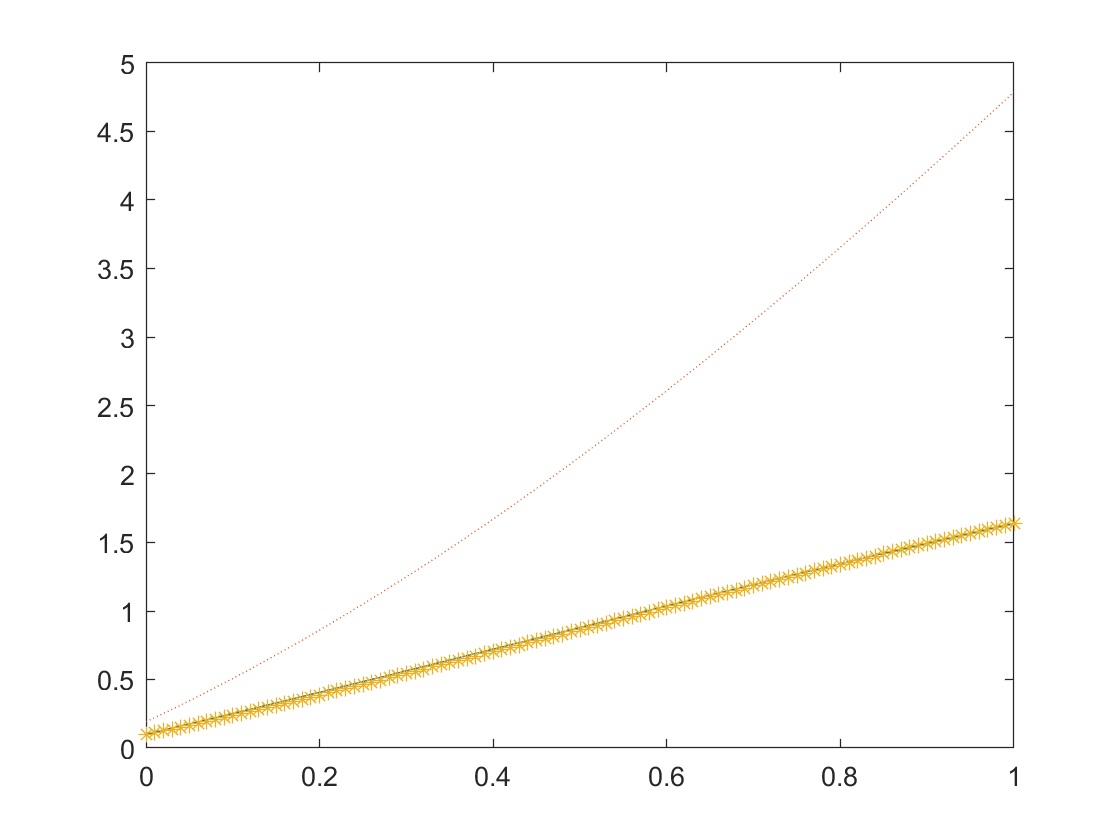}
\\[-1em]
\caption{{\small The (blue) solid line represents the original signal $f$, the dotted (red) plot represents the signal $g$ affected by the multiplicative noise, the asterisk (yellow) plot represent the denoised signal by (\ref{denoising-method}) for $\mu_2$, with $n=10$ (on the top) and $n=30$ (on the bottom). The maximum reconstruction error are $0.0658$ and $0.0202$, respectively.}} \label{fig2}
\end{figure}
\begin{figure}[H]
\centering
\includegraphics[scale=0.27]{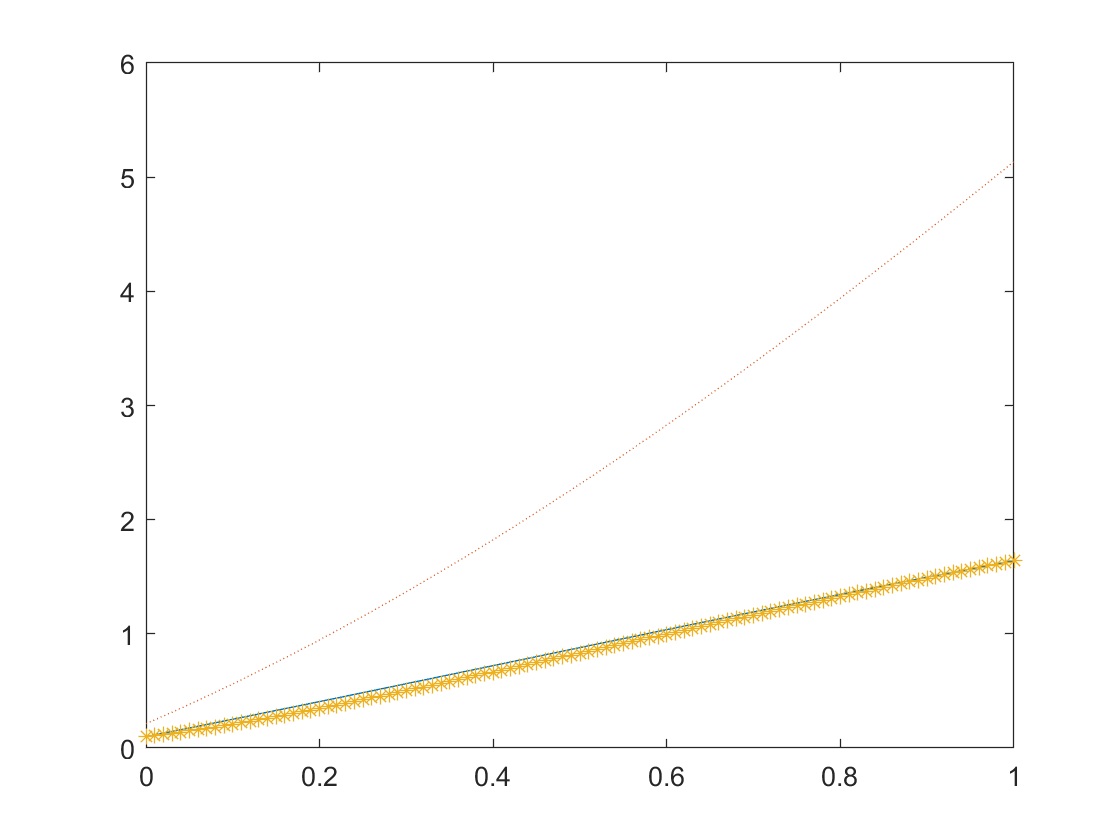}
\end{figure}
\begin{figure}[h]
\centering
\includegraphics[scale=0.25]{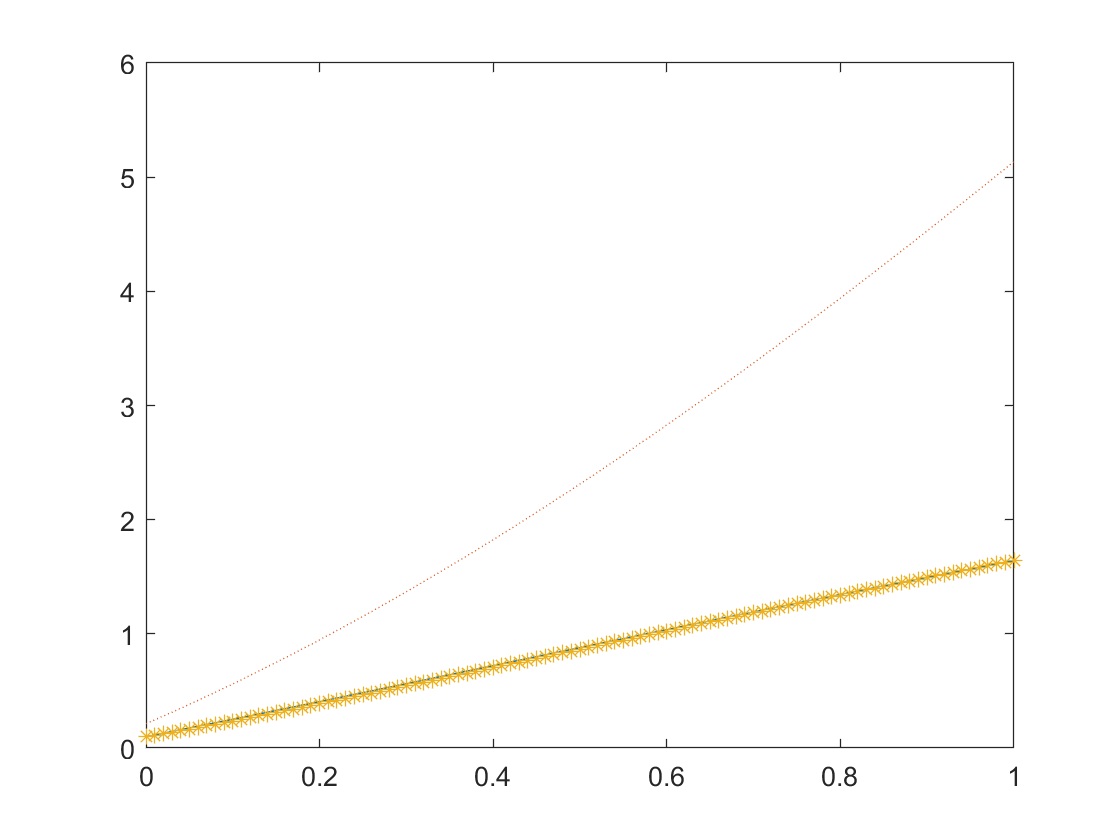}
\\[-1em]
\caption{{\small The (blue) solid line represents the original signal $f$, the dotted (red) plot represents the signal $g$ affected by the multiplicative noise, the asterisk (yellow) plot represent the denoised signal by (\ref{denoising-method}) for $\mu_3$, with $n=10$ (on the top) and $n=30$ (on the bottom). The maximum reconstruction errors are $0.0622$ and $0.0191$, respectively.}} \label{fig3}
\end{figure}
Obviously, the above experiment represents only a {\em toy model} which serves to introduce a possible application of the new logarithmic-type operators; the main idea was to exploit the logarithm-preservation property besides the approximation properties of ${\mathscr L}_n$ to linearize a nonlinear source of noise. In future works, we aim to improve the above simple model, for instance, in order to deal with multivariate signals to implement a despeckle algorithm for remote sensing images (see, e.g. \cite{despeckle1,speckle2,WL2025speckle}).

\section*{Acknowledgements}
{\small The authors would like to deeply thank the anonymous Referees for their valuable comments and useful suggestions that improved the paper. The authors are members of the Gruppo Nazionale per l'Analisi Matematica, la Probabilit\`a e le loro Applicazioni (GNAMPA) of the Istituto Nazionale di Alta Matematica (INdAM), of the network RITA (Research ITalian network on Approximation), and of the UMI (Unione Matematica Italiana) group T.A.A. (Teoria dell'Approssimazione e Applicazioni). 
}

\section*{Funding}
{\small The authors L. Angeloni and D. Costarelli have been partially supported within the (1) "National Innovation Ecosystem grant ECS00000041 - VITALITY", funded by the European Union - Next-GenerationEU under the Italian Ministry of University and Research (MUR), (2) PRIN 2022 PNRR: ``RETINA: REmote sensing daTa INversion with multivariate functional modeling for essential climAte variables characterization", funded by the European Union under the Italian National Recovery and Resilience Plan (NRRP) of NextGenerationEU, under the Italian Ministry of University and Research (Project Code: P20229SH29, CUP: J53D23015950001), and (3) 2025 GNAMPA-INdAM Project "MultiPolExp: Polinomi di tipo esponenziale in assetto multidimensionale e multivoco" (CUP E5324001950001).}

\section*{Conflict of interest/Competing interests}
{\small The authors declare that they have no conflict of interest and competing interests.}

\section*{Availability of data and material and Code availability}
{\small Not applicable.}

\bibliographystyle{plain}

\begin{thebibliography}{32}

\bibitem{acar2017szasz}
T. Acar, A. Aral, D. C\'ardenas-Morales and P. Garrancho. Sz{\'a}sz--Mirakyan type operators which fix exponentials. \textit{Results in Mathematics}, 72:1393--1404, 2017.

\bibitem{acar2017gonska}
T. Acar, A. Aral and H. Gonska. On Sz{\'a}sz--Mirakyan operators preserving $e^{2ax}$, $a>0$. \textit{Mediterranean Journal of Mathematics}, 14:1--14, 2017.

\bibitem{acar2020gamma}
T. Acar, M. Mursaleen and S. N. Deveci. Gamma operators reproducing exponential functions. \textit{Advances in Difference Equations}, 2020:1--13, 2020.

\bibitem{Acu2022a}
A. M. Acu, A. Aral and I. Ra\c{s}a. Generalized Bernstein Kantorovich operators: Voronovskaya type results, convergence in variation. \textit{Carpathian Journal of Mathematics}, 38:1--12, 2022.

\bibitem{Altomare2010}
F. Altomare. Korovkin-type theorems and approximation by positive linear operators. \textit{Surveys in Approximation Theory}, 6:92--164, 2010.

\bibitem{AltoCam}
F. Altomare and M. Campiti. \textit{Korovkin-type Approximation Theory and its applications.} Walter de Gruyter \& Co., Berlin, New York, 1994.

\bibitem{Angeloni2024}
L. Angeloni and D. Costarelli. Approximation by exponential-type polynomials. \textit{Journal of Mathematical Analysis and Applications}, 532:127927, 2024.

\bibitem{Angeloni2025}
L. Angeloni, D. Costarelli and C. Darielli. Approximation processes by multidimensional Bernstein-type exponential polynomials on the hypercube. \textit{Revista de la Real Academia de Ciencias Exactas, F\'isicas y Naturales - Serie A: Matem\'aticas}, 119(28):1--15, 2025.

\bibitem{BerType}
A. Aral, D. C{\'a}rdenas-Morales and P. Garrancho. Bernstein-type operators that reproduce exponential functions. \textit{Journal of Mathematical Inequalities}, 12:861--872, 2018.

\bibitem{Aral2019} 
A. Aral, M. L. Limmam and F. Ozsara\c{c}. Approximation properties of Sz{\'a}sz--Mirakyan--Kantorovich type operators. \textit{Mathematical Methods in the Applied Sciences}, 42(16):5233-5240, 2019. 

\bibitem{despeckle1} 
F. Argenti, A.  Lapini, T. Bianchi and L. Alparone. A tutorial on speckle reduction in synthetic aperture radar images. \textit{IEEE Geoscience and remote sensing magazine}, 1(3):6--35, 2013.

\bibitem{Aubert} 
G. Aubert and J.F. Aujol. A variational approach to removing multiplicative noise. \textit{SIAM journal on applied mathematics}, 68(4):925-946, 2008.

\bibitem{BABO1964}
B. Bajsanski and R. Bojanic. A note on approximation by Bernstein polynomials. \textit{Bulletin of the American Mathematical Society}, 70:675--677, 1964.

\bibitem{bardaro2017}
C. Bardaro, L. Faina and I. Mantellini. A generalization of the exponential sampling series and its approximation properties. \textit{Mathematica Slovaca}, 67(6):1481--1496, 2017.

\bibitem{bardaro2019}
C. Bardaro, I. Mantellini and G. Schmeisser. Exponential sampling series: convergence in Mellin--Lebesgue spaces. \textit{Results in Mathematics}, 74:1--20, 2019.

\bibitem{bede2016}
B. Bede, L. Coroianu and S. G. Gal. \textit{Approximation by max-product type operators.} Springer, Cham, Switzerland, 2016.

\bibitem{Bohman}
H. Bohman. On approximation of continuous and analytic functions.
\textit{Arkiv f{\"o}r Matematik}, 2:43--46, 1952--54.

\bibitem{Devore1972}
R. A. DeVore. \textit{The Approximation of Continuous Functions by Positive Linear Operators}, volume 293 of \textit{Lecture Notes in Mathematics}. Springer, Berlin, 1972.

\bibitem{ConApp}
R. A. DeVore and G. G. Lorentz. Constructive Approximation.
Springer-Verlag Berlin, Heidelberg, New York, 1 edition, 1993.

\bibitem{finta2023} 
Z. Finta. King operators which preserve $x^j$. \textit{Constructive Mathematical Analysis}, 6(2):90-101, 2023. 

\bibitem{GACA2010}
P. Garrancho and D. C{\'a}rdenas-Morales. A converse of asymptotic formulae in simultaneous approximation. \textit{Applied Mathematics and Computation}, 217:2676--2683, 2010.

\bibitem{gonska2009king}
H. Gonska, P. Pit\c{u}l and I. Ra\c{s}a. General King-Type Operators.
\textit{Results in Mathematics}, 53:279--286, 2009.

\bibitem{gupta2022} 
V. Gupta, A. Aral and F. Özsara\c{s}. On semi-exponential Gauss–Weierstrass operators. \textit{Analysis and Mathematical Physics}, 12(5):111, 2022. 

\bibitem{gupta2024conv}
V. Gupta. Convergence of operators based on some special functions.
\textit{Revista de la Real Academia de Ciencias Exactas, F{\'i}sicas y Naturales - Serie A: Matem{\'a}ticas}, 118(3):99, 2024.

\bibitem{gupta2024new}
V. Gupta. New operators based on Laguerre polynomials. \textit{Revista de la Real Academia de Ciencias Exactas, F{\'i}sicas y Naturales - Serie A: Matem{\'a}ticas}, 118(1):19, 2024.

\bibitem{gupta2020}
V. Gupta and G. Agrawal. Approximation for modification of exponential type operators connected with $x(x+1)^2$. \textit{Revista de la Real Academia de Ciencias Exactas, F{\'i}sicas y Naturales - Serie A: Matem{\'a}ticas}, 114(3):158, 2020.

\bibitem{gupta2018}
V. Gupta and A. Aral. A note on Sz{\'a}sz--Mirakyan--Kantorovich type operators preserving $e^{-x}$. \textit{Positivity}, 22:415--423, 2018.



\bibitem{kadak2022}
U. Kadak. Multivariate neural network interpolation operators.
\textit{Journal of Computational and Applied Mathematics}, 414:114426, 2022.

\bibitem{KAST1966}
S. J. Karlin and W. J. Studden. \textit{Tchebycheff Systems.} Interscience, New York, 1966.

\bibitem{king}
J. King. Positive linear operators which preserve $x^2$. \textit{Acta Mathematica Hungarica}, 99(3):203--208, 2003.

\bibitem{Korovkin1953}
P. P. Korovkin. Convergence of linear positive operators in the spaces of continuous functions (Russian). In \textit{Doklady Akademii Nauk SSSR (NS)}, volume 90, page 961, 1953.

\bibitem{Morigi}
S. Morigi and M. Neamtu. Some results for a class of generalized polynomials. \textit{Advances in computational mathematics}, 12:133--149, 2000.


\bibitem{ozsarac2023} 
F. Özsara\c{c}, A. M. Acu, A. Aral and I. Ra\c{s}a. On the modification of Mellin convolution operator and its associated information potential. \textit{Numerical Functional Analysis and Optimization}, 44(11):1194-1208, 2023. 

\bibitem{ozsarac2022}
F. Özsara\c{c}, V. Gupta and A. Aral. Approximation by some Baskakov--Kantorovich exponential-type operators. \textit{Bulletin of the Iranian Mathematical Society}, 48(1):227--241, 2022.

\bibitem{speckle2} 
R. Ren, Z. Guo, Z. Jia, J. Yang,  N. K. Kasabov and C. Li. Speckle noise removal in image-based detection of refractive index changes in porous silicon microarrays. \textit{ Scientific Reports}, 9(1):15001, 2019.

\bibitem{Mond}
O. Shisha and B. Mond. The degree of convergence of sequences of linear positive operators. \textit{Proceedings of the National Academy of Sciences}, 60:1196--1200, 1968.

\bibitem{Seelamantula2015} 
C. S. Seelamantula and T. Blu. Image denoising in multiplicative noise. \textit{IEEE International Conference on Image Processing (ICIP)}, 1528-1532, 2015.

\bibitem{Topuz2024} 
C. Topuz, F. Ozsara\c{c} and A. Aral. On the generalized Mellin integral operators. \textit{Demonstratio Mathematica}, 57(1):20230133, 2024. 

\bibitem{WL2025speckle}
J. Wei and X. Liao. Dynamical Threshold-Based Fractional Anisotropic Diffusion for Speckle Noise Removal. \textit{IEEE Transactions on Image Processing}, 34:2826--2839, 2025.

\bibitem{Convexity}
Z. Ziegler. Linear approximation and generalized convexity. \textit{Journal of Approximation Theory}, 1:420--443, 1968.

\end{thebibliography}

\end{document}